\documentclass[11pt]{amsart}
\usepackage{amssymb, amsmath, amsfonts, amsthm, mathtools,hyperref,dsfont} 
\usepackage{latexsym}
\usepackage[usenames,dvipsnames]{color}

\allowdisplaybreaks[4]

\newtheorem{thm}{Theorem}[section] 
\newtheorem{dfn}[thm]{Definition}
\newtheorem{rmk}[thm]{Remark}
\newtheorem{rmks}[thm]{Remarks}
\newtheorem{cor}[thm]{Corollary}
\newtheorem{prop}[thm]{Proposition}

\newtheorem{lem}[thm]{Lemma}

\def\cyclic{\mathop{\kern0.9ex{{+}
\kern-2.2ex\raise-.28ex\hbox{\Large\hbox
{$\circlearrowright$}}}}\limits}

\def\buildrel#1_#2^#3{\mathrel{\mathop{\kern 0pt#1}\limits_{#2}^{#3}}}

\newcommand{\Aut}{{\rm Aut}}

\newcommand{\Id}{{\rm Id}}

\newcommand{\eps}{\varepsilon}

\newcommand{\bO}{{\bf\Omega}}

\newcommand{\C}{\mathbb C}

\newcommand{\Z}{\mathbb Z} 

\newcommand{\R}{\mathbb R}

\newcommand{\Q}{\mathbb Q} 
 
\newcommand{\N}{\mathbb N}

\renewcommand{\k}{{\mathfrak{k}}{}}

\renewcommand{\k}{{\mathfrak{k}}{}}
\newcommand{\CO}{{\mathcal O}{}}

\newcommand{\CS}{\mathcal S}

\newcommand{\CL}{{\mathcal L}{}} 
 
\newcommand{\CD}{\mathcal D}

\newcommand{\CH}{\mathcal H}
\newcommand{\CB}{\mathcal B}
\newcommand{\CU}{\mathcal U}
\newcommand{\CF}{\mathcal F}

\newcommand{\vf}{\varphi}

\renewcommand{\k}{{\bf k}}

\def\cref#1{Corollary~\ref{#1}}

\usepackage{fancyhdr}
\title{Quantization of the affine group of a local field}
 
 \date{}
 \author{Victor Gayral}
 \author{David Jondreville}
\begin{document}

\begin{abstract}
For a non-Archimedean local field which is not
of characteristic $2$, nor an extension of $\Q_2$, we construct a pseudo-differential calculus covariant under a unimodular subgroup of the affine group of the  field. Our phase space is   a quotient group of the covariance group. Our main result 
is a generalization on that context of the Calder\'on-Vaillancourt estimate. Our construction can be thought as the non-Archimedean version of Unterberger's Fuchs calculus
 and our methods are mainly based on Wigner functions and on coherent states transform.

\end{abstract}

\maketitle

{\Small{\bf Keywords:}  Equivariant quantization,
Local fields, $p$-adic pseudo-differential analysis, $p$-adic Fuchs calculus, Coherent states, Wigner functions, Calder\'on-Vaillancourt estimate}

\tableofcontents
 \section{Introduction}
 The present work fits in a research program \cite{BG,BGT,GJ1,GJ3} where we aim to generalise Rieffel's theory  \cite{Ri} of deformation of $C^*$-algebras for 
 group actions, in the case of more general groups than the Abelian group $\R^d$.
  In \cite{BG} we have  studied the case of   K\"ahlerian Lie groups of negative sectional curvature, \cite{BGT} extends Rieffel's construction in the 
 super-symmetric case and \cite{GJ1} deals with non-Archimedean local fields instead of the real line. The present paper concerns the non-Abelian and non-Lie situation of the affine 
 group of a non-Archimedean local field. However, we only present here
 the pseudo-differential calculus  part of the construction, the rest will be published elsewhere \cite{GJ3}. The point  is that  the generalisation
 of the Weyl calculus is an  important subject on its own and it has generated an important activity over the last decades. See for instance 
 \cite{AE,AU1,AU2,Bechata,BB,BC1,BC2,FGBV,FR,GGBV,GS,GB,Haran,Manchon,P,RT,Un84,UnUn,VGB,W}  (which is by no mean exhaustive).
 Concretely, we present here  what we call the $p$-adic Fuchs calculus  and  prove a crucial estimate concerning $L^2$-boundedness.

 Our pseudo-differential calculus is covariant under  a subgroup of the affine group of a non-Archime\-dean  local field $\k$ (of characteristic different from  $2$ and which is
 not an extension of $\Q_2$). Explicitly, the covariance group is  the semidirect product $G_n:= U_n\ltimes\k$ where   $U_n$ is  the subgroup of the multiplicative group of
 units (the unit sphere of $\k$) given by the principal units of order $n\in\mathbb{N}\setminus\{0\}$  acting
 on the additive group $\k$ by  dilations. 
 Our phase space $X_n $ is, as usual, 
 an homogeneous space for $G_n$. In fact, $X_n$ is a quotient group of $G_n$  and has the form $U_n\ltimes\Gamma_n$ where $\Gamma_n$ is discrete and 
 Abelian.   Last, our configuration space is the compact group $U_n$ and this choice is dictated by the representation theory of 
  the locally compact group $G_n$.
  
   That the phase space is also endowed with a compatible group structure is  essential to generalise Rieffel's machinery. Indeed,
the starting point of Rieffel
   deformation of $C^*$-algebras for  actions of a (locally compact) group $G$ is the data of an associative, noncommutative and $G$-equivariant
    product on a suitable space of functions on the group $G$. Such a product is precisely  what equivariant quantization yields (using  the composition law of symbols),
    provided the phase-space $X$ is endowed with a group structure compatible with the action of $G$.

To define our pseudo-differential calculus, we construct a quantization map (from distributions on the phase-space $X_n$ to operators acting on  functions on the configuration
space $U_n$) by mimicking,    in this $p$-adic setting,  Unterberger's  Fuchs calculus   \cite{Un84}.  In the latter construction, 
 the covariance group and 
 the phase space   both coincide with the connected component of the affine group of the real line. But since a non-Archimedean  local field is totally disconnected,
 there is no obvious $p$-adic version of the Fuchs calculus, at least from a topological perspective. 
  In fact, the way the construction of  \cite{Un84} should be adapted to the $p$-adic world, is dictated by analytical considerations. Indeed,  what we really need is an open
   subgroup $U$
  of the multiplicative group $\k^\times$ such that the square mapping
  $$
  U\to U, \quad u\mapsto u^2,
  $$
  and the hyperbolic sine type mapping
  $$
  U\to V,\quad u\mapsto u-u^{-1},
  $$
  (where $V$ is a suitable open subset of $\k$) are homeomorphisms. 
  This is why we need to work with the compact group $U_n$ and  this is also  why we need to
   exclude characteristic $2$ and extensions of $\Q_2$.
  The existence of  a square root function and of an inverse hyperbolic sine function (on the configuration space)
  is in fact the main common point with the classical Fuchs calculus.

Before going on, we should first explain from an abstract perspective, the type of pseudo-differential calculus we are interested in. So, let $G$ be a (second countable
Hausdorff) locally compact topological group, $H$ be a closed subgroup  and let $X=G/H$ be the associated
 homogeneous space. Fix now $(\pi,\CH_\pi)$ a projective unitary irreducible continuous representation of $G$
which we assume to be square integrable (see \cite{Aniello} for the projective case). 
(Projectivity is important to get a non-trivial construction even when $G$ is Abelian, e.g$.$ \cite{GJ1}.) Fix also $\Sigma$
a  self-adjoint operator on $\CH_\pi$, whose  domain $D$ is  $G$-invariant, such that the map 
$G\to\CH_\pi$, $g\mapsto\Sigma\pi(g)\vf$ ($\vf\in D$)  is locally bounded and such that  $[\Sigma,\pi(H)]=0$. In this case, the operator valued function  $g\mapsto
\pi(g)\Sigma\pi(g)^*$ is invariant under right translations on $H$ and therefore defines a map on the quotient $X$. We  denote this map by $\Omega$.
From this, we have a well defined quantization map
\begin{equation}
\label{la-belle}
\bO:C_c(X)\to \mathcal L(D,\CH_\pi),\qquad f\mapsto \int_X f(x)\,\Omega(x)\,dx,
\end{equation}
from compactly supported continuous functions on $X$ to  $\mathcal L(D,\CH_\pi)$, 
 the set of densely defined operators with domain $D$. In the above formula,  $dx$ is a $G$-invariant measure on $X$
(that we assume to exist) and the integral is understood in the weak sense
on $\CH_\pi\times D$. By construction, this quantization map is $G$-covariant:
\begin{equation}
\label{lab-elle}
\pi(g)\bO(f)\pi(g)^*=\bO(f^g)\quad\mbox{where}\quad f^g(x)=f(g^{-1}.x),\quad g\in G,\;x\in X.
\end{equation}
To generalise Rieffel's construction, we need  more properties.
First, we need to assume  the quantization map to be invertible and that its image forms an algebra.
In such a case, we have a well defined  composition law
of symbols, that is a non-formal  $\star$-product:
$$
f_1\star f_2:=\bO^{-1}\big(\bO(f_1)\bO(f_2)\big).
$$
  (Omitting all the analytical details, to get an analogue of Rieffel's construction, this non-formal and associative 
  $\star$-product will eventually  be extended from functions on $X$ to elements of a $C^*$-algebra $A$
endowed with a strongly continuous and isometric action $\alpha$ of the group $X$ by the rule:
$a_1\star^\alpha a_2:=\bO^{-1}\big(\bO(\alpha(a_1))\bO(\alpha(a_2))\big)(e)$,  $a_1,a_2\in A$.
Here, for $a\in A$,  $\alpha(a)$ is the $A$-valued continuous function on the group $X$ given by $[x\mapsto\alpha_x(a)]$ and $e$ is the neutral element of $X$.)

At least for type $I$ groups, the most natural  assumption is  that $\bO$ extends as unitary operator
from $L^2(X)$ to the Hilbert space of Hilbert-Schmidt operators on $\CH_\pi$. When this property holds, we speak about \emph{unitary quantization}.
It is this condition that imposes (at least when $G$ is not unimodular) to work
with unbounded $\Sigma$.    In fact, we can easily replace the Hilbert space of Hilbert-Schmidt operators by the GNS space of a normal semifinite faithful weight on a
von Neumann subalgebra of $\CB(\CH_\pi)$ containing $\bO(C_c(X))$ without changing many things. However,  in all the examples we know
yet, it is the Hilbert-Schmidt operators that 
appear.

Last but not least, we need $X$ to be endowed with a group structure in such a way that the covariance property \eqref{lab-elle} gives also covariance
of the $\star$-product (the composition law of symbols) for the left action of the group $X$.
If this condition is not satisfied, then  there is no way to extend Rieffel's construction.
The simplest  situation where this property occurs is when $G$ possesses a closed subgroup $G_o$ acting simply transitively on $X$ and this is what happens  in \cite{BG}. 
Under the  identification $G_o\simeq X$, the topological space $X$ 
becomes a locally compact group and the $G$-invariant measure $dx$ becomes a left Haar measure on the group $X$. Hence, the map \eqref{la-belle} defines a 
$X$-covariant unitary quantization on the group $X$ and (with $\lambda$ the left regular representation)  the covariance property \eqref{lab-elle} now reads 
$\pi(x)\bO(f)\pi(x)^*=\bO(\lambda_xf)$, $x\in X\simeq G_o\subset G$. 
At the level of the $\star$-product, it immediately implies left covariance:
$\lambda_x(f_1\star f_2)=\lambda_x(f_1)\star \lambda_x(f_2)$, $x\in X$.
More generally, this property holds when the closed subgroup $H$ is normal in $G$ (this is the situation we came across in this paper). Hence, $X=G/H$ is a quotient group. 
However, it is not true in general that the quotient group $X$ acts on $\CH_\pi$ so that the covariance property is a slightly more delicate question. In fact, since $[\Sigma,\pi(H)]=0$,
it is easy so see that the quotient group $X$ acts by conjugation on the von Neumann algebra generated by the operators $\bO(f)$, $f\in C_c(X)$. Therefore, 
the covariance property \eqref{lab-elle}  reads in that case $\pi(g)\bO(f)\pi(g)^*=\bO(\lambda_xf)$ where $g\in G$ and  $x=gH\in X$. But at
the level of the $\star$-product, we still obtain  left covariance for the regular action of the group $X$:
$\lambda_x(f_1\star f_2)=\lambda_x(f_1)\star \lambda_x(f_2)$.

  The paper is organized as follows. In Section \ref{aprem}, we set the ingredients  that will be needed in our construction.
  In Section \ref{PAFC}, we define the $p$-adic Fuchs calculus and prove its basic properties. The final Section \ref{CV} contains our main result, namely
  an extension of the Calder\'on-Vaillancourt Theorem for the $p$-adic Fuchs calculus.

\section{Preliminaries}
\label{aprem}

\subsection{The phase space and its covariance group}
Let ${\bf k}$ be a non-Archimedean local field. In characteristic zero, $\k$ is isomorphic to   a finite 
degree extension of a field
of $p$-adic numbers $\Q_p$. In positive characteristic, $\k$ is isomorphic to
a field of Laurent 
series $\mathbb F_q((X))$ with coefficients in a finite field $\mathbb F_q$.
We let $|.|_{\bf k}$ be the ultrametric absolute value  (the restriction to dilations of the module 
function),   $\CO_{\bf k}$ be the ring of integers of $\k$ (the closed --and open-- unit ball), $\varpi$  be a 
generator of its unique maximal ideal (called a uniformizer) and $\CO_{\k}^{\times}$ be 
the multiplicative group of units (the --open-- unit sphere). 
We also denote by $U_1:=1+\varpi\CO_{\k}$ the group of principal units and, more generally, 
 for an  integer $n >1$ we denote by $U_n:=1+\varpi^n\CO_{\k}$ the higher principal units group.
We let $q$ be the cardinality of the residue field $\CO_{\k}/\varpi\CO_{\k}$. The uniformizer $\varpi$ 
then satisfies $|\varpi|_{\k} = q^{-1}$. We will also let $p$ be the characteristic of the residue field. 
It satisfies $q=p^f$ for an integer $f\in\N\setminus\{0\}$.

For technical reasons that we will explain soon,
we will be forced to exclude two cases: when the characteristic of $\k$ is $2$ and when $\k$ is an 
extension of  $\Q_2$. The next Proposition--certainly well 
known to the experts--will help to understand why  we do need such  a restriction.

\begin{prop}
\label{exclud}
The following assertions are equivalent:
\begin{enumerate}
\item $\k$ is of characteristic different from $2$ and it is not an extension of $\Q_2$;
\item $q$ is not divisible by $2$;
\item $2$ belongs to the group of units $\CO_{\k}^{\times}$;
\item $-1$  does not belong to the group of principal units $U_1$.
\end{enumerate}
\end{prop}

\begin{proof}
$(1) \Leftrightarrow (2)$ This is obvious since $q=p^f$. 

$(2) \Rightarrow (3)$  This is also obvious. Indeed, in characteristic zero, if $p \neq 2$, then $2$ is a 
unit of $\Q_p$, hence a unit of any finite degree extension of $\Q_p$. Similarly, in  characteristic different from two, $2$ is a unit of 
$\mathbb F_\ell((X))$.

$(3) \Rightarrow (4)$ By contraposition, suppose that $-1 \in U_1$. Then, there exists $z 
\in \CO_{\k}$ such that $-1 = 1+\varpi z$, so that $ 2 = -\varpi z \in \varpi\CO_{\k}$. Hence 
$|2|_\k=q^{-1}|z|_\k<1$.

$(4) \Rightarrow (1)$ By contraposition, suppose that $\k$ is a finite extension of $\Q_2$ or that the 
characteristic of $\k$ is $ 2$. In characteristic zero, we have $-1 = 
\sum_{n\in\N}2^n \in 1+\varpi\CO_{\k}$. In characteristic $2$, we have 
$-1 = 1 \in 1+ \varpi\CO_{\k}$. 
\end{proof}

\noindent
{\it From now on, we  assume the local field $\k$ to satisfy one the equivalent conditions of
Proposition  \ref{exclud}.}

\quad

Note that $U_n$ is  an open and  compact  multiplicative subgroup of $\k^\times$ and that $U_n$
 acts
  by dilations on $\k$. We can therefore  consider the  semidirect product $G_n:=U_n  \ltimes  \k$, that we
  view as a  subgroup of  the affine group $\k^\times\ltimes\k$,
   with group law
  \begin{equation}
\label{grouplaw}  (x,t).(x',t')  := (xx', x'^{-1}t+t')\,,\quad x,x'\in\k^\times\,,\;t,t'\in\k.
 \end{equation}
 The group $G_n$ will be the \emph{covariance group} of our pseudo-differential calculus and the subgroup $U_n$  will play the
role of the \emph{configuration space}.
 The role of the \emph{phase space} will be played by  the 
 homogeneous space $X_n :=  G_n/H_n$, associated to the closed subgroup 
 $H_n := \{1\} \times \varpi^{-n}\CO_{\k}$. Note that since $H_n$ is a normal  subgroup of $G_n$,   
$X_n$ is naturally endowed with a group structure. In fact, $X_n$ is  isomorphic to  $U_n \ltimes \Gamma_n$,
where $\Gamma_n := \k  /{\varpi^{-n}\CO_{\k}}$ is discrete and Abelian.
As a locally compact  group, the phase space $X_n$ is not compact, nor 
discrete and non-Abelian. We will denote the elements of $X_n$
by pairs $(u,[t])$, where $u\in U_n$ and $[t]:=t+\varpi^{-n}\CO_\k\in\Gamma_n$,
 so that the group law of $X_n$ is given by
$$
(u,[t]).(u',[t']) \; := \; (uu', [u'^{-1}t+t']).
$$
To simplify notation, we will frequently denote an element of $G_n$ by $g:=(u,t)$ and an element
of $X_n$ by $[g]:=(u,[t])$. Also, $e:=(1,0)$ will denote the neutral element of $G_n$
and $[e]:=(1,[0])$, the one of $X_n$.
Note also that the action of $U_n$ on 
$\Gamma_n$, $(u,[t])\mapsto[ut]$, is isometric for the quotient metric:
$$
d_{\Gamma_n}\big([t_1],[t_2]\big):=d_\k\big( t_1,t_2+\varpi^{-n}\CO_\k\big)=\inf_{x\in\varpi^{-n}\CO_\k}|t_1-t_2+x|_\k.
$$
($d_{\Gamma_n}$ is nondegenerate because $t_2+\varpi^{-n}\CO_\k$ is closed and it
satisfies
the triangle inequality because $2\in\CO_\k^\times$.) That we are in the very specific situation where the phase space $X_n$ 
is also endowed with a group structure will be very important for future developments of the present work \cite{GJ3}
(see the introduction).
This explains why we will carry on studying some group-theoretical properties of $X_n$, even if we simply need the group structure
of $G_n$ here.

We  normalize the Haar measure $dt$ of $(\k,+)$
so that ${\rm Vol}(\CO_\k)=1$. Since $U_n\subset \CO_\k^\times$, we have 
$|u|_\k=1$ for all $u\in U_n$,
so that the Haar measure of $\k^\times$ (given by  $|t|_\k^{-1}dt$) restricted to $U_n$
is just $dt$.  With that choice we have ${\rm Vol}(U_n)=q^{-n}$.
The Haar measure of  $G_n$, denoted by $dg$, is chosen to be the restriction to
$U_n\times \k$ of  the Haar measure of $\k\times\k$. The $G_n$-invariant measure of  $X_n$
(which is also the Haar measure of $X_n$),
denoted by $d[g]$, is chosen to be the product of the Haar measure 
of $U_n$ by the counting measure on $\Gamma_n$.
In particular,  both groups $G_n$ and $X_n$ are unimodular. 

We fix once for all  a unitary character $\Psi$ of the additive group $\k$ 
which is required to be trivial on 
$\CO_\k$ but not on $\varpi^{-1}\CO_\k$. (The conductor of $\Psi$ is therefore $\CO_\k$ itself.)
Recall that   $\widehat\CO_\k$,  the Pontryagin dual of the compact-open additive group $\CO_\k$, 
is naturally  identified with the discrete group $\k/\CO_\k$. Indeed, if  for $G$ a locally
compact Abelian group and for $H$ an open subgroup of $G$, we let  $A(\widehat G,H)$
be the annihilator  of $H$ in the Pontryagin dual $\widehat G$, we then have by \cite[Lemma
24.5]{HR} that $\widehat H\simeq \widehat G/A(\widehat G,H)$. The identification
$\widehat\CO_\k\simeq \k/\CO_\k$ follows since
(by self-duality relative to $\Psi$) we have 
$A(\widehat \k,\CO_\k)\simeq \CO_\k$.  In the same vein, we also
deduce by \cite[Theorem 23.25]{HR} that $\widehat \Gamma_n\simeq A(\widehat \k,\varpi^{-n}\CO_\k)
\simeq\varpi^n\CO_\k$.

As explained in the introduction, the reason why we are working with the higher units group $U_n$ for the configuration
space (and not the full units group $\CO_\k^\times$ or
the whole dilations group $\k^\times$)  is because we need  a well defined square root mapping to 
handle our $p$-adic Fuchs  calculus. The following statement results in an  essential way on the assumption that  $2$ belongs to $\CO_\k^\times$. 
\begin{prop}
\label{sigma}
The square function $\sigma:U_n\to U_n$, $u\mapsto u^2$, and the hyperbolic sine type function
$\phi:U_n\to \varpi^n\CO_\k$, $u\mapsto u-u^{-1}$ are
  $C^1$-homeomorphisms (see \cite[p.77]{Schikhof}) with $|\sigma'|_\k=|\phi'|_\k=1$.
\end{prop} 
\begin{proof}
Let $u,v\in U_n$. Then we have
$$
\sigma(u)-\sigma(v)=(u-v)(u+v), \quad \mbox{and} \quad
\phi(u)-\phi(v)=(u-v)(1+u^{-1}v^{-1}).
$$ 
Note then that both $u+v$ and  $1+u^{-1}v^{-1}$ belong to $2+\varpi^n\CO_\k$. Since
$2\in\CO_\k^\times$ (by assumption on $\k$) and $|\varpi|_\k<1$, we get $|u+v|_\k=|1+u^{-1}v^{-1}|_\k=1$ and thus:
$$
|\sigma(u)-\sigma(v)|_\k=|u-v|_\k=|\phi(u)-\phi(v)|_\k.
$$
 Hence $\sigma$ and $\phi$ are isometries. But they are also  surjections since $U_n$ and $\varpi^n\CO_\k$
 are compact metric spaces (for the induced metric of $\k$). This  shows that
$\sigma$ and $\phi$ are homeomorphisms.
That they are  of class $C^1$ is obvious since for all  $u\in U_n$, we have
\begin{align*}
\sigma'(u)&=\lim_{(x,y)\to (u,u)}\frac{\sigma(x)-\sigma(y)}{x-y}=\lim_{(x,y)\to (u,u)}x+y=2u\in2+\varpi^n\CO_\k,\\
\phi'(u)&=\lim_{(x,y)\to (u,u)}\frac{\phi(x)-\phi(y)}{x-y}=\lim_{(x,y)\to (u,u)}1+x^{-1}y^{-1}=1+u^{-2}\in 2+\varpi^n\CO_\k,
\end{align*}
by continuity of the addition, multiplication and inversion in $\k^\times$.
This also immediately implies that $|\sigma'(u)|_\k=|\phi'(u)|_\k=1$ for all $u\in U_n$.
\end{proof}

\begin{rmk}{\rm
It is easy to find in the literature a proof that the principal unit group of $\mathbb Q_p$, $p\ne 2$,
consists only on squares (see for instance \cite[Page 18]{Serre} or \cite[Proposition 3.4]{Sally}).
Our approach is based instead on general arguments of compact  metric spaces.}
\end{rmk}

\begin{dfn}
In the following, we define the square root function $U_n\to U_n$, $u\mapsto u^{1/2}$, as the 
reciprocal mapping
of the square function $\sigma:U_n\to U_n$.
\end{dfn}

Note that in characteristic zero, the group $X_n$ is isomorphic to the semidirect 
product $ \CO_{\k} \ltimes_{\alpha_n} \widehat{\CO}_{\k}$. Here
 $\widehat\CO_\k$  is identified with the discrete group $\k/\CO_\k$ 
 (whose elements are denoted
by $t+\CO_\k$) and 
 the extension homomorphism  $\alpha_n\in{\rm Hom}
\big(\CO_\k,\Aut(\widehat\CO_\k)\big)$ is given by
$\alpha_{n,x}( t+\CO_\k):=\exp_\k(\varpi^n x)t+\CO_\k$, for $(x,t+\CO_\k)\in\CO_\k\times \widehat\CO_\k$. Here,
with $p$ the characteristic of the residue field $\CO_\k/\varpi\CO_\k$, we denote by 
$$
\exp_\k :E_p:=\big\{t\in\k\,:\, |t|_\k< p^{1/(1-p)}\big\}\to 1+ E_p, \quad t\mapsto\sum_{k\geq 0} \frac{t^k}{k!},
$$ 
 the (isometrical) exponential map of $\k$ (which is defined if and only if $\k$ is of characteristic zero). Then, it is easy to show that the map
  $$ \CO_{\k} \ltimes_{\alpha_n} \widehat{\CO}_{\k} \rightarrow X_n ,
  \quad (x,  t+\CO_\k) \mapsto \big(\exp_{\k}(\varpi^nx), [\varpi^{-n}t]\big),$$
  provides the desired group isomorphism (which is well defined for any $n\in\N^*$ since $p>2$).  
  In the same vein, we  also have $G_n\simeq \CO_\k\ltimes_{\beta_n} \k$ where $\beta_{n,x}(t):=\exp_\k(\varpi^n x)t$, 
for $(x,t)\in\CO_\k\times \k$.
  Now,  the isomorphisms $X_n\simeq \CO_{\k} \ltimes_{\alpha_n} \widehat{\CO}_{\k}$ and $G_n\simeq \CO_\k\ltimes_{\beta_n} \k$
  allow to present our groups as semidirect products with fixed subgroups but with varying extension automorphisms. It thus permits us   to obtain a classification
  result  in case  of characteristic zero:

 \begin{prop}
 If $\k$ is of  characteristic zero and $n\ne m$ then, $G_n$ is not isomorphic to $G_m$ and
 $X_n$ is not isomorphic to  $X_m$.
 \end{prop}
 \begin{proof}
 We first prove the second claim.
 Set $Y_n:= \CO_{\k} \ltimes_{\alpha_n}  \widehat{\CO}_{\k}$.
 By the  isomorphism $X_n\simeq Y_n$, it suffices to proof that $Y_n\simeq Y_m$ if and only if $n=m$.
 Observe first that the centralizer of $ \widehat{\CO}_{\k}$ in $\CO_{\k}$ is trivial. Recall then that the
 topological automorphism 
 group of $\widehat \CO_\k$ is isomorphic to $\CO_\k^\times$ (which follows by a 
 straightforward generalization of the arguments given for $\mathbb Q_p$ in
 \cite[Example (e) page 434]{HR} and the fact that the topological 
 automorphism group of $\widehat \CO_\k$ is isomorphic to those of $\CO_\k$
 \cite[Theorem 26.9]{HR}) hence Abelian. Now, by \cite[Corollary 2]{kuz},
$Y_n\simeq Y_m$ if and only if $\alpha_n(\CO_\k)$ is conjugate to 
$\alpha_m(\CO_\k)$ in $\Aut(\widehat \CO_\k)$. But since the latter is Abelian, 
$Y_n\simeq Y_m$ if and only if $\alpha_n(\CO_\k)=\alpha_m(\CO_\k)$. Since finally 
$$
\alpha_n(\CO_\k)=
\Big\{\widehat \CO_\k\to \widehat \CO_\k,\quad  t+\CO_\k\mapsto \exp_\k(\varpi^n x)t+\CO_\k
\;:\; x\in \CO_\k\Big\},
$$
we have $\alpha_n(\CO_\k)=\alpha_m(\CO_\k)$ if and only if $m=n$, hence the result.

The case of $G_n$ is almost identical using the fact that $\Aut(\k)\simeq \k^\times$  (hence Abelian too).
 \end{proof}

\begin{rmk}
{\rm
In positive characteristic we conjecture that $X_n\simeq X_m$ and $G_n\simeq G_m$ if and only if $n=m$,  even  if there is no  group isomorphism
between $U_n$ and $\CO_{\k}$ (see \cite[page 72]{Schikhof}), hence even if there is certainly no
group isomorphism between $X_n$ and $\CO_{\k} \ltimes \widehat{\CO}_{\k}$
and between $G_n$ and $\CO_{\k} \ltimes \k$.
}
\end{rmk}

With our  choice of normalization for the Haar measure of $\k$
and with the nontrivial  previously fixed basic character $\Psi$, the Fourier transform
$$
\big(\CF_\k\, f\big )(s):=\int_\k f(t)\,\Psi(st)\,dt,
$$
becomes  unitary on $L^2(\k)$.
More generally, when an Abelian group
$H$ is not self-dual (for instance $\CO_\k$, $\CO_\k^\times$, $U_n$ and $\Gamma_n$), we define the 
Fourier transform
$$
\CF_H:L^2(H)\to L^2(\widehat H),\quad
f\mapsto \Big[\chi\mapsto \int_H f(h) \,\chi(h) \,dh\Big],
$$
and normalize the Haar measure of $H$ (or of $\widehat H$) in such a way that $\CF_H$
becomes a unitary operator.

We will frequently identify a function $\tilde{f} \in L^1(\Gamma_n)$ with the function $f \in L^1(\k)$ 
invariant under  translations in $\varpi^{-n}\CO_{\k}$.
With our choices of normalization, we therefore have 
 \begin{equation} 
 \label{intg} 
 \sum_{[t] \in \Gamma_n}  \tilde{f}([t])    =     q^{-n} \, \int_{\k} f(t) \, dt.
\end{equation}
In particular, if $f \in L^1(\k)$ is constant on the cosets of $\varpi^{-n}\CO_{\k}$ then $\CF_\k(f)$
is supported on $\varpi^n\CO_\k$ and using the identification $\widehat\Gamma_n\simeq \varpi^n
\CO_\k$, we get
\begin{align}
\label{ident-fourier}
\CF_\k\, f=q^n\,\CF_{\Gamma_n}\tilde f.
\end{align}

 We will also frequently use the following substitution formula  (see \cite[p.287]{Schikhof} for details):
\begin{equation*}
   \int_U   f \circ \vf(t)  \, |\vf'(t)|_{\k} \, dt=\int_V   f(t)  \, dt.
 \end{equation*}
Here, $U, \, V$ are compact open subsets of $\k$, $f : V \rightarrow \C$ is $L^1$ and 
$\vf : U \rightarrow V$ is a $C^1$-homeomorphism (see \cite[p.287]{Schikhof}) such that 
$\vf'(t) \neq 0$ for all $t \in U$. 
For instance, it follows by Proposition \ref{sigma}, that if $f : U_n \rightarrow \C$ 
and $h:\varpi^n\CO_\k\to\C$ are $L^1$ then
\begin{equation}
\label{change}
 \int_{U_n}  \, f(u^2)  \, du=  \int_{U_n} \,  f(u)  \, du
 \quad\mbox{and}\quad \int_{U_n}h(u-u^{-1})\,du=\int_{\varpi^n\CO_\k} h(x)\,dx.
 \end{equation}

We denote by $\CD(\k)$ the space of Bruhat-test functions 
  on $\k$ and by $\CD'(\k)$ the dual space of Bruhat-distributions  \cite{Bruhat}, endowed with the weak
   dual topology. We denote by $\langle T|\vf\rangle$ the evaluation of a distribution $T\in\CD'(\k)$ on a test
    function $\vf\in\CD(\k)$.  
      Since $\k$ is totally disconnected, $\CD(\k)$ coincides with the space of locally
   constant compactly supported functions on $\k$. Equivalently, $f\in \CD(\k)$ if and only if there
   exists $n,m\in\Z$ such that $f$ is supported on $\varpi^n\CO_\k$ and $f$
   is invariant under  translations in $\varpi^m\CO_\k$.
Recall also that the Fourier transform $\CF_\k$ is an homeomorphism of $\CD(\k)$ and of
   $\CD'(\k)$. Moreover, if  $T\in\CD'(\k)$ is supported on $\varpi^n\CO_\k$  then
   $\CF_\k(T)$ is invariant under  translations in $\varpi^{-n}\CO_\k$
   and vice versa.  Topologically, $\CD(\k)$ is the inductive limit of the sequence of Banach spaces
    consisting of functions supported in $\varpi^n\CO_\k$, constant on the cosets of
   $\varpi^m\CO_\k$ and endowed with the uniform norm.

   Similarly, we let $\CD(U_n)$, $\CD(X_n)$ 
be
   the spaces of Bruhat-test functions  on the groups $U_n$ and $X_n$. 
   Since $U_n$ is compact, $\CD(U_n)$ is the space of locally constant functions on $U_n$ and $\CD(X_n)$
 is the space of functions on $X_n$ which are locally
   constant in the variable $u\in U_n$ and have finite support in the variable $[t]\in \Gamma_n$.
  Here, $f$ is locally constant in $U_n$, is equivalent to the existence of $m\geq n$ such that 
$f$ is invariant under  dilations in $ U_m$. But  this is also equivalent to 
  $f$ is invariant under  translations in 
  $\varpi^m\CO_\k$. Indeed, if $f(uv)=f(u)$ for all $u\in U_n$ and $v\in U_m$ then for $x\in\varpi^m\CO_\k$,
  we have $f(u+x)=f(u(1-u^{-1}x))=f(u)$ since $u^{-1}\in\CO_\k^\times$ and thus $1-u^{-1}x\in U_m$. Conversely, if
  $f(u+x)=f(u)$ for all $u\in U_n$ and $x\in\varpi^m\CO_\k$ then for $v\in U_m$ we have $f(uv)=f(u+(v-1)u)=f(u)$
  since $(v-1)u\in\varpi^m\CO_\k$.

  A sequence $\{f_n\}_{n\in\N}$
  converges  in $\CD(U_n)$ if the $f_n$ are constant  on the cosets of $U_m$ in $U_n$  for
  a fixed integer $m\geq n$ and if the sequence converges uniformly on $U_n$. Similarly,
  a sequence $\{f_n\}_{n\in\N}$
  converges  in $\CD(X_n)$ if the $f_n$ are constant on the cosets of $U_m$ in $U_n$ (in the variable $u\in U_n$) for
  a fixed integer $m\geq n$, if they are supported (in the $[t]$-variable) in fixed finite subset of 
  $\Gamma_n$
   and if the sequence converges uniformly on $X_n$.
   (For  details see for instance the exposition given
  in \cite{SallyTaibleson}.) Note finally that the spaces $\CD(\k)$, $\CD'(\k)$ and $\CD(U_n)$, $\CD'(U_n)$
  also coincide with the Schwartz space and the space of tempered distributions as defined
  in \cite[section 9]{Bruhat}. However, here we reserve the notations $\CS(U_n)$ and $\CS'(U_n)$
  (as well as $\CS(X_n)$ and $\CS'(X_n)$) to denote other spaces of functions/distributions.
   
Since  the groups  $U_n$ and $X_n$
   are separable, all these spaces of test-functions and of distributions are nuclear
   (see \cite[Corollary 5 page 53]{Bruhat} and the remark that follows it) and satisfies the Schwartz
   kernel Theorem: given a continuous linear map $A$ from $\CD(G_1)$ to $\CD'(G_2)$
   ($G_1,G_2$ are $U_n$ or $X_n$) there exists a unique $T\in\CD'(G_1\times G_2)$ such that
   for all $\vf_j\in G_j$ the have $ \langle A(\vf_1)|\vf_2\rangle= \langle T|\vf_1\otimes\vf_2\rangle$
   (see \cite[Corollary 2 page 56]{Bruhat}). However $\CD(U_n)$ and $\CD(X_n)$ fail to be Fr\'echet
   (they are only LF-spaces). 

\subsection{The representation theory of  $G_n$ and $X_n$}
\label{RG}
With $\Psi$ the fixed unitary character of $({\bf  k},+)$ 
(trivial on $\CO_{\k}$ but not on $\varpi^{-1}\CO_{\k}$) and for $\theta \in \k^\times$, 
we set $\Psi_\theta(t):=\Psi(\theta t)$.
We define the  following representation  $\pi_{\theta}$ of the covariance group $G_n$ on the Hilbert space $L^2(U_n)$ of square integrable functions on the configuration
space:
\begin{equation}
\label{Utheta}
\pi_{\theta}(u,t)\varphi(u_0) :=\Psi_{\theta}( u_0^{-1}ut) \,
\varphi\big(u^{-1}u_0\big).
 \end{equation}
 
 Note that when $\theta \in \varpi^n\CO_{\k} \setminus \{0\}$,  the unitary operator 
 $\pi_{\theta}(a,t)$ only  depends on the class of $t$ in $\Gamma_n$. Hence, it defines a
  unitary representation  of the quotient group $X_n$ that we denote by $\widetilde{\pi}_{\theta}$. 
  It is (almost) immediate to see that $\pi_{\theta}$ is equivalent to ${\rm Ind}_\k^{G_n}(\Psi_\theta)$, the representation $\Psi_\theta\in\widehat \k$ of $\k$  induced to $G_n$.
  Analogously, if $\theta \in \varpi^n\CO_{\k} \setminus \{0\}$ 
then $\widetilde \pi_{\theta}$ is equivalent to the induced representation
${\rm Ind}_{\Gamma_n}^{G_n}(\Psi_\theta)$ where $\Psi_\theta$ is viewed as an element of $\widehat\Gamma_n$ in a natural  way.

From  Mackey's theory \cite{Mackey}, 
we know that all irreducible representations of $G_n=U_n\ltimes \k$
(and of $X_n=U_n\ltimes \Gamma_n$ too) are of the form 
${\rm Ind}_{G_\Theta}^{G_n}(\rho_\Theta \otimes \Psi_\Theta) $.  Here $\Theta$ is an orbit of 
the dual action of $U_n$ on $\widehat \k$  and
$\Psi_\Theta$ is any point in the orbit $\Theta$. Also,
$G_\Theta$ is the subgroup of $G_n$  given by $H_\Theta\ltimes\k$, where $H_\Theta$ is
 the stabilizer
(the little group) of $\Psi_\Theta$. Last, $\Psi_\Theta$ is  extended
to a character (a one-dimensional representation) of $G_\Theta$ in an obvious way
and  $\rho_\Theta$ is an irreducible representation of $H_\Theta$ extended to $G_\Theta$ too.
 For the trivial orbit $\Theta=\{0\}$, we get $H_\Theta = U_n$ and  $G_\Theta = G_n$. 
  Thus, ${\rm Ind}_{G_\Theta}^{G_n}(\rho_\Theta \otimes \Psi_\Theta ) $ is  one-dimensional and we 
  disregard
  such representations (which are of zero Plancherel measure). In contrast, if $\Theta$ is not trivial,
we get $H_{\Theta} = \{1 \}$  and $G_\Theta = \{ 1 \} \ltimes \k$. 
Therefore, the induced representation we get is ${\rm Ind}_\k^{G_n}(\Psi_\Theta)$ which is 
$\pi_{\theta}$ with $\Psi_\Theta=\Psi_\theta$. Since moreover 
${\rm Ind}_\k^{G_n}(\Psi_\Theta)$ is equivalent to ${\rm Ind}_\k^{G_n}(\Psi_{\Theta'})$ if and only if 
$\Theta=\Theta'$, we deduce that $\pi_{\theta}$ is equivalent to $\pi_{\theta'}$
if and only if $\theta'\theta^{-1}\in U_n$.
Of course,
similar considerations apply for the quotient group $X_n$. Hence we have obtained:

\begin{prop}
Any infinite dimensional unitary irreducible representation of $G_n$ is  unitarily equivalent to the representation $\pi_{\theta}$ for some $\theta\in\k^\times$. Moreover,
$\pi_{\theta}$ is equivalent to $\pi_{\theta'}$ if and only if $\theta'\theta^{-1}\in U_n$.
Analogously, any  infinite dimensional  unitary irreducible representation of $X_n$ is  unitarily equivalent to the representation $\widetilde \pi_{\theta}$  for some 
$\theta \in \varpi^n\CO_{\k} \setminus \{0\}$. Moreover,
$\widetilde \pi_{\theta}$ is equivalent to $\widetilde \pi_{\theta'}$ if and only if $\theta'\theta^{-1}\in U_n$.
 \end{prop}
\begin{rmk}
\label{restrict}
{\rm
There are infinitely many inequivalent irreducible representations of $G_n$ and $X_n$. But
if we restrict to $\theta \in \CO_{\k}^{\times}$ there are exactly $(q-1)  q^{n-1}$ 
inequivalent irreducible representations of $G_n$. Indeed, they are labelled by the elements of 
$\CO_{\k}^{\times}/U_n$, and $[\CO_{\k}^{\times}: U_n ] = [\CO_{\k}^{\times}:U_1 ] 
\times  [U_1 : U_2]\times\dots\times  [U_{n-1} : U_n]= (q-1)  q^{n-1}$ (which follows since $\CO_{\k}^{\times}/U_1$ is isomorphic to
the multiplicative group of the residue field and $U_{n-1} /U_n$ is isomorphic to the additive group 
of the residue field). Observe also that $\pi_{-\theta}$ is the contragredient representation 
of $\pi_{\theta}$.
}
\end{rmk}

Next we prove the square integrability of the  representations $\pi_{\theta}$ and 
$\widetilde{\pi}_{\theta}$:
\begin{prop}  
\label{squareint} 
For $\theta \in \k ^\times$ (resp$.$ for $\theta \in \varpi^n\CO_{\k} \setminus \{0\}$)
 the  representation $\pi_{\theta}$ of $G_n$ 
 (reps$.$ $\widetilde{\pi}_{\theta}$ of $X_n$) is square integrable. 
\end{prop}
\begin{proof} 
We first give the proof for the representation $\pi_{\theta}$ of $G_n$. 
 Let $\varphi_1, \, \varphi_2\in\CD(U_n)$. 
 Then, by unimodularity of $U_n$ we get
\begin{align*}
\int_{G_n }    | \langle \varphi_1 , \pi_{\theta}(g) \varphi_2 \rangle     |^2 \, dg
& = \int_{U_n\times\k } \Big |  \int_{U_n} \overline{\varphi_1}(u_0) \, \Psi_{\theta}(u_0^{-1}ut) \, 
 \varphi_2(u^{-1}u_0) du_0 \Big |^2 \, dudt\\
  &=
 \int_{U_n\times\k} \Big |  \int_{U_n}  \overline{\varphi_1}(u_0^{-1})  \, \Psi_{\theta}(u_0ut) \,  
 \varphi_2(u^{-1}u_0^{-1}) du_0 \Big |^2  \, dudt\\
 &=\int_{U_n\times\k}  \Big  |  \int_{U_n}  \overline{\varphi_1}(u_0^{-1}u)  \, 
 \Psi_{\theta}(u_0t)  \, \varphi_2(u_0^{-1}) du_0  \Big |^2 \, dudt.
 \end{align*}
 Let $ \mathds{1}_{U_n}$ be the characteristic function of $U_n\subset \k$. 
Introducing the family of  compactly supported locally constant functions on $\k$
 given by
$$
F_u := \big [u_0 \mapsto  \mathds{1}_{U_n}(u_0) \,  \overline{\varphi_1}(u_0^{-1}u)   \, \varphi_2(u_0^{-1}) 
\big ], \quad u\in U_n,
$$
 we get
\begin{align*}
\int_{G_n }     | \langle \varphi_1 , \pi_{\theta}(g) \varphi_2 \rangle   |^2 \, dg  & =
  \int_{U_n\times\k}    \big | \CF_{\k} ( F_u  )(\theta t)    \big |^2 \, dudt. 
\end{align*}
Then, the dilation $t \mapsto \theta^{-1}t$ and the Plancherel formula for the self-dual group $\k$ yield
\begin{align*}
\int_{G_n }     | \langle \varphi_1 , \pi_{\theta}(g) \varphi_2 \rangle   |^2 \, dg  &  = |\theta|_{\k}^{-1}  
\int_{U_n^2}      | \overline{\varphi_1}(u_0^{-1}u) \,  \varphi_2(u_0^{-1})   |^2 \, du du_0,
\end{align*}
so that, by the dilation  $u \mapsto  u_0u$ followed by the inversion $u_0 \mapsto u_0^{-1}$,
we finally deduce
\begin{equation}\label{formsquare} \int_{G_n }     | \langle \varphi_1 , 
\pi_{\theta}(g) \varphi_2 \rangle   |^2 \, dg   = |\theta|_{\k}^{-1} \,  \| \vf_1\|_2^2 \, \| \vf_2\|_2^2,
\end{equation}
which completes the proof of the first part since $\CD(U_n)$ is dense in $L^2(U_n)$.
The proof for the
representation  $\widetilde{\pi}_{\theta}$ of $X_n$ is entirely similar using instead of an  integral
on $\k$ a sum over the countable group $\Gamma_n$ and applying   the relations \eqref{intg}
and \eqref{ident-fourier}
to  be able to  use the Plancherel formula for $\k$.
\end{proof}

A direct consequence of Proposotion \ref{squareint} is that  $\pi_\theta$  is weakly contained in $\lambda$, the left regular representation  of $G_n$. 
Now, a straightforward computation shows 
that for $f\in L^2(G_n)$, the operator kernel of  $\pi_\theta(f)$ is given by $(\Id\otimes\CF_\k f)(u_1u_2^{-1},\theta u_2^{-1})\in L^2(U_n\times U_n)$. This easily entails the
following formula for the
Hilbert-Schmidt norm:
$$
\|\pi_\theta(f)\|_2^2=\int_{G_n} \mathds{1}_{\theta U_n}(t)\big|\Id\otimes\CF_\k f\big|^2(u,t)\,dudt.
$$
From this formula, we deduce that the orthogonal projector $P_\theta$ onto the closed left-invariant subspace of $L^2(G_n)$  in which $\lambda$ induces a representation 
unitarily equivalent to $\pi_\theta$ is given by:
$$
P_\theta=\Id\otimes\CF_\k^{-1}\,\mathds{1}_{\theta U_n}\,\CF_\k,\quad \theta\in\k^\times.
$$
Similar considerations apply for the representation $\tilde\pi_\theta$, $\theta\in \varpi^n\CO_{\k} \setminus \{0\}$, of the quotient group $X_n$ and
the associated projector reads:
$$
\tilde P_\theta=\Id\otimes\CF_{\Gamma_n}^{-1}\,\mathds{1}_{\theta U_n}\,\CF_{\Gamma_n},\quad \theta\in \varpi^n\CO_{\k} \setminus \{0\}.
$$

For any square integrable representation, and associated to a mother wavelet $\vf\in L^2(U_n)\setminus\{0\}$, we have a  unitary coherent-state transform
which immediately follows from the results of Duflo and Moore \cite{Duflo-Moore}:
$$
L^2(U_n)\to L^2(G_n),\quad \psi\mapsto\Big[g\in G_n\mapsto |\theta|_{\k}^{1/2}||\vf||_2^{-1} \langle \pi_{\theta}(g)\vf,\psi\rangle\Big].
$$
Indeed:
\begin{cor}
\label{ident}
For every $\vf_1, \, \vf_2  \in L^2(U_n)$, $\theta\in\k^\times$ and non-zero $\vf \in L^2(U_n)$, we have
\begin{equation*}
\langle \vf_1,\vf_2 \rangle = \dfrac{ |\theta|_{\k}}{||\vf||_2^2} \, \int_{G_n} \langle \vf_1, 
\pi_{\theta}(g) \vf \rangle \, \langle \pi_{\theta}(g) \vf, \vf_2 \rangle \, dg. 
\end{equation*}
Provided that 
$\theta \in \varpi^n\CO_{\k} \setminus \{0\}$, the same formula holds with $\widetilde \pi_{\theta}$ instead of $\pi_{\theta}$ 
and $\int_{X_n}$ instead of $\int_{G_n}$.
\end{cor}
\begin{proof}
This is just the polarized version of the identity \eqref{formsquare}, viewed as 
an hermitian  form in the variable $\vf_1\in L^2(U_n)$ and setting $\vf_2=\vf$.
\end{proof}

\section{The p-adic Fuchs calculus}
\label{PAFC}
\subsection{The  quantization map  and the Wigner functions}
In this subsection, our aim is to construct a $G_n$-equivariant pseudo-differential calculus on
the phase space $X_n$. Our construction is an adaptation to the $p$-adic setting of
Unterberger's Fuchs calculus \cite{Un84}. We will eventually be forced to restrict ourself to the case where 
the representation  parameter is a unit
$\theta\in\CO_\k^\times$.
In view of Remark \ref{restrict}, it means that we only consider $(q-1)q^{n-1}$ inequivalent 
representations of $G_n$. But at the beginning, we may consider the parameter $\theta$  to be in $\CO_\k\setminus\{0\}$.

We let $\Sigma$ be the  operator  on $L^2(U_n)$ implementing the group inversion:
\begin{equation*}
\Sigma \, \varphi(u) :=  \varphi({u}^{-1}).
 \end{equation*}
By unimodularity, $\Sigma$ is bounded and this  highly simplifies the general discussion given in the introduction.

\begin{lem}
\label{omegexp}
Assume that $\theta\in\CO_\k\setminus\{0\}$.
Then the map $G_n \rightarrow \CB(L^2(U_n))$, $g \mapsto \pi_{\theta}(g)  \Sigma \pi_{\theta}^*(g)$ is 
constant on the  cosets of $H_n$ in $G_n$. Therefore, it defines a map 
$$
\Omega_{\theta} : X_n=G_n/H_n \rightarrow \CB(L^2(U_n)),
$$
which, with $\phi:U_n\to \varpi^n\CO_\k$ as given in Proposition \ref{sigma} and $[g] = (u,[t]) \in X_n$, reads:
\begin{equation*}
 \Omega_{\theta}([g])\vf(u_0) := \Psi_{\theta}\big( \phi(uu_0^{-1}) t\big)    \varphi(u^2 u_0^{-1}).
 \end{equation*}
\end{lem}

\begin{proof}
Let $(u,t) \in G_n$, $ u_0 \in U_n$ and $ \varphi \in L^2(U_n)$. Then we have
\begin{align*}
 \pi_{\theta}(u,t) \, \Sigma \, \pi^*_{\theta}(u,t) \vf(u_0) 
&= \Psi_{\theta}(u_0^{-1}ut) \, \big(\Sigma \, \pi^*_{\theta}(u,t) \vf\big)(u^{-1}u_0) \\
&= \Psi_{\theta}(u_0^{-1}ut)  \,  \Psi_{\theta}(-u_0u^{-1}t) \, \vf(u^2u_0^{-1}) 
= \Psi_{\theta}( \phi( uu_0^{-1} )t)   \, \vf(u^2u_0^{-1}).
\end{align*}
That the map $g \mapsto \pi_{\theta}(g)  \Sigma \pi_{\theta}^*(g)$ is invariant under translations in 
$H_n=\{1\}\times\varpi^{-n}\CO_\k$ follows from the fact that   $\phi( uu_0^{-1} )\in \varpi^{n}\CO_\k$
and that $\Psi_\theta$ is trivial on $\CO_\k$ since $\theta\in\CO_\k\setminus\{0\}$.
\end{proof}


We are now ready to introduce our $G_n$-equivariant    quantization map  on $X_n$:
 
\begin{dfn}
For $\theta\in\CO_\k\setminus\{0\}$, we
let  ${\bf \Omega}_{\theta} :  L^1(X_n) \rightarrow \CB(L^2(U_n))$  be the continuous linear map
given  by   
\begin{equation*}  
{\bf \Omega}_{\theta}(f)  := |\theta|_\k\,q^{n} \, \int_{X_n} f([g])  \, \Omega_{\theta}([g]) \,  d[g]. 
\end{equation*}
\end{dfn}
For $(g,[g'])\in G_n\times X_n$, we have by construction
$$
\pi_{\theta}(g)  \Omega_{\theta} ([g'])  \pi^*_{\theta}(g)= \Omega_{\theta} ([gg'])
=\Omega_{\theta} (g.[g']).
$$
 This immediately 
implies that the quantization map ${\bf \Omega}_{\theta}$ is $G_n$-equivariant: If  $g\in G_n$
and $f\in L^1(X_n)$, setting $f^g([g'])=f(g^{-1}.[g'])=f([g^{-1}g'])$, we have by $G_n$-invariance of the Haar measure of $X_n$ :
 \begin{equation}
  \label{equiv}
   \pi_{\theta}(g) \, {\bf \Omega}_{\theta} (f) \, \pi^*_{\theta}(g) = {\bf \Omega}_{\theta}(f^g).
     \end{equation}
Note that with $\lambda$ the left regular representation of the group $X_n$, we also have
 \begin{equation*}
   \pi_{\theta}(g) \, {\bf \Omega}_{\theta} (f) \, \pi^*_{\theta}(g) = {\bf \Omega}_{\theta}
   \big( \lambda_{[g]} f\big).
     \end{equation*}
Provided $\theta \in \varpi^n\CO_{\k}$,  the above equivariance relation also holds for the 
representation $\widetilde \pi_{\theta}$ of $X_n$. But this
observation is pointless since we will soon be forced
to work with $\theta\in\CO_\k^\times$. Moreover, since $\Omega_\theta([g])$ is self-adjoint, we get
$\bO_\theta(f)^*=\bO_\theta(\overline f)$.

Our main tool  for the analysis of our pseudo-differential calculus are 
 the matrix coefficients  of $\Omega_\theta$, viewed as an operator valued  function on $X_n$.
In analogy with ordinary quantum mechanics, we call them   the Wigner functions.

 \begin{dfn}
 Let $\theta\in\CO_\k\setminus\{0\}$.
For  $\varphi_1$, $\varphi_2 \in L^2(U_n)$ we let $W_{\varphi_1, \, \varphi_2}^{\theta}$
be the Wigner function:
\begin{equation*}
W_{\varphi_1, \, \varphi_2}^{\theta}: X_n\to \C,\quad [g]\mapsto 
 \big \langle \vf_1, 
\Omega_{\theta}([g]) \vf_2 \big \rangle. 
\end{equation*}
\end{dfn} 

One of the main properties of the Wigner functions is  regularity. 
To compare the next result with analogues in Archimedean pseudo-differential calculi (Weyl, Fuchs...) we have to remember
 that in the $p$-adic world, Bruhat's notions of smooth compactly supported function and of Schwartz function are identical.

 \begin{lem}
 \label{wignerbruhat}
 Let $\theta\in\CO_\k\setminus\{0\}$.
 Then the bilinear map
$$
(\vf_1, \, \vf_2 )\mapsto W^{\theta}_{\vf_1, \, \vf_2},
$$
sends continuously $\CD(U_n)\times \CD(U_n)$ to  $\CD(X_n)$.
 \end{lem}
 
 \begin{proof}
To prove that $W^{\theta}_{\vf_1, \, \vf_2}\in\CD(X_n)$ when  $\vf_1, \, \vf_2 \in \CD(U_n)$,
we need to show that the map $U_n\to \C$, $u\mapsto W^{\theta}_{\vf_1, \, \vf_2}(u,[t])$ is locally 
constant and that the map $\Gamma_n\to \C$, $[t]\mapsto W^{\theta}_{\vf_1, \, \vf_2}(u,[t])$ has
finite support.
 
So let $(u,[t]) \in X_n$ and $\vf_1, \, \vf_2 \in \CD(U_n)$. Without lost of generality, we may assume that 
$\vf_1$ and $\vf_2$ possess the same invariance domain, say $U_m$ for some $m\geq n$.  We then have:
\begin{align}
\label{wig1}
\big  \langle \varphi_1 , \Omega_{\theta}(u,[t]) \varphi_2 \big \rangle =  \int_{U_n}  
\overline{\varphi_1}(u_0)  \, \Psi_{\theta}\big(\phi(uu_0^{-1} ) t\big)   \,   \varphi_2(u^2u_0^{-1}) \, du_0. 
\end{align}
 Then, since $\phi(u_0^{-1}) = -\phi(u_0)$, the dilation $u_0 \mapsto u u_0$ entails
\begin{equation*}
  \big  \langle \varphi_1 , \Omega_{\theta}(u,[t]) \varphi_2 \big \rangle   = \int_{U_n}  
  \overline{\varphi_1}(u_0 u) \, \overline{\Psi_{\theta}}\big( \phi( u_0)t\big)   \,   
   \varphi_2(uu_0^{-1}) \, du_0. 
  \end{equation*}
Take now $b\in U_m$ (the invariance domain of $\vf_1$ and of $\vf_2$). Since 
$\overline{\varphi_1}(u_0 ub)=\overline{\varphi_1}(u_0 u)$ and $\varphi_2(ubu_0^{-1})=
\varphi_2(uu_0^{-1})$ for all $u,u_0\in U_n$, we  deduce that  
$u\mapsto W^{\theta}_{\vf_1, \, \vf_2}(u,[t])$ is constant in the cosets of $U_m$ in $U_n$.

Next, observe that the (second) substitution formula \eqref{change} entails:  
\begin{align*}
 \label{inv} 
 \big  \langle \varphi_1 , \Omega_{\theta}(u,[t]) \varphi_2 \big \rangle  &  = 
  \int_{\varpi^n\CO_{\k} }  \overline{\varphi_1}\big(\phi^{-1}(z_0)  u\big) \, 
  \overline{\Psi_{\theta}}(z_0t)   \,    
 \varphi_2\big(u( \phi^{-1}(z_0))^{-1}\big) \, dz_0. 
  \end{align*}
Consider then the family of (uniformly in $u\in U_n$) compactly  supported functions on $\k$:
$$
f_{u} :=\Big[ z_0 \mapsto  |\theta|_{\k}^{-1}  \, \mathds{1}_{\theta\varpi^n\CO_{\k}}(z_0)  \,  
\overline{\varphi_1}\big(\phi^{-1}(\theta^{-1}z_0) u\big) \, 
\varphi_2\big(u(\phi^{-1}(\theta^{-1}z_0))^{-1}\big)\Big]
,\quad u\in U_n.
$$
Since $\phi^{-1}:\varpi^n\CO_\k\to U_n$ is continuous (it is an isometry) and $\vf_1,\vf_2$ are locally 
constant, it follows that the functions $z_0\mapsto \overline{\varphi_1}\big(\phi^{-1}(\theta^{-1}z_0) 
u\big)$
and $z_0\mapsto  \varphi_2\big(u(\phi^{-1}(\theta^{-1}z_0))^{-1}\big)$ are locally constant too, 
with invariance domain independent of $u\in U_n$.
Hence, $f_u$ is locally constant with invariance domain independent of $u\in U_n$. This means that $f_u\in\CD(\k)$ and thus 
$\CF_{\k}^{-1}( f_{u})\in\CD(\k)$.
But since we have
\begin{equation}
\label{wig-four}
\big  \langle \varphi_1 , \Omega_{\theta}(u,[t]) \varphi_2 \big \rangle    = 
\CF_{\k}^{-1}( f_{u})(t),
\end{equation} 
it follows that the map $\k\to\C$,
$t \mapsto  \big  \langle \varphi_1 , \Omega_{\theta}(u,[t]) \varphi_2 \big \rangle$, 
 has compact support (uniformly in $u\in U_n$).
So, the induced map $\Gamma_n\to\C$, 
$[t] \mapsto  \big  \langle \varphi_1 , \Omega_{\theta}(u,[t]) \varphi_2 \big \rangle$,  
has finite support (uniformly in $u\in U_n$).

Continuity of the map $\CD(U_n)\times \CD(U_n)\to\CD(X_n)$,
$(\vf_1,  \vf_2 )\mapsto W^{\theta}_{\vf_1,  \vf_2}$ follows from two  facts.
First,  we have shown that the invariance domain of $u\mapsto W^{\theta}_{\vf_1, \, \vf_2}
(u,[t])$ and
the support of the map $[t]\mapsto W^{\theta}_{\vf_1,  \vf_2}
(u,[t])$  only depends on the invariance domain of $\vf_1$ and $\vf_2$ but not on $\vf_1$ and $\vf_2$
themselves.  Second, 
we have the obvious estimate:
$$
\big|W^{\theta}_{\vf_1,  \vf_2}(u,[t])\big|= \big|\big \langle \vf_1, 
\Omega_{\theta}([g]) \vf_2 \big \rangle\big|
\leq \,{\rm Vol}(U_n)\,\|\Omega_{\theta}([g])\|\,\|\vf_1\|_\infty\,\|\vf_2\|_\infty=
q^{-n}\,\|\vf_1\|_\infty\,\|\vf_2\|_\infty.
$$
Take then two sequences $\{\vf_{\ell}^j\}_{\ell\in\N}\subset \CD(X_n)$, $j=1,2$,  which converge
to zero in $\CD(X_n)$. This means  that all the $\vf^j_{\ell}$ 
($\ell\in\N$ and  $j=1,2$) have the same invariance domain and the same support, and that $\{\vf^j_{\ell}\}_{\ell\in\N}$ 
($j=1,2$)
converges to zero in sup-norm. Hence all the terms of the double sequence 
$\big\{W^{\theta}_{\vf^1_{\ell_1},  \vf^2_{\ell_2}}\big\}_{\ell_1,\ell_2\in\N}\subset\CD(X_n)$
have the same invariance domain (in the variable $u\in U_n$) and the same support (in the variable
$[t]\in\Gamma_n$) and this double sequence converges to zero in sup-norm. 
 \end{proof}

Take now $f\in\CD(X_n)$ and $\vf_1,\vf_2\in\CD(U_n)$. Using Fubini's Theorem, we have
\begin{align*}
\big \langle \vf_1,  {\bf \Omega}_{\theta}(f) \vf_2  \big \rangle=
|\theta|_\k\,q^n\int_{U_n\times X_n} f([g])\,
\overline{\vf_1}(u) \,(\Omega_\theta([g])\vf_2\big)(u)\,du\,d[g]= |\theta|_\k\,q^n
\int_{X_n} f([g])\,W_{\varphi_1, \, \varphi_2}^{\theta}([g])\,d[g].
\end{align*}
This, together with Lemma \ref{wignerbruhat}, allows to extend the quantization map 
${\bf \Omega}_{\theta}$   from $L^1$-functions to distributions. In the following,
$\mathcal L\big(\CD(U_n),\CD'(U_n)\big)$ denotes the space of continuous linear maps
from $\CD(U_n)$ to $\CD'(U_n)$. Recall also that we denote by $\langle T|\vf\rangle$ the evaluation of a Bruhat 
distribution $T\in\CD'(U_n)$ on a Bruhat test function $\vf\in \CD(U_n)$.
 
 \begin{prop} 
 \label{extend}
 Let $\theta\in\CO_\k\setminus\{0\}$.
 The map ${\bf \Omega}_{\theta} : L^1(X_n) \rightarrow \CB(L^2(U_n))$ extends uniquely
 to a continuous linear map
 $ {\bf \Omega}_{\theta}:\CD'(X_n)\to
\mathcal L\big(\CD(U_n),\CD'(U_n)\big)$. This extension is given by 
 $$
 \big\langle{\bf \Omega}_{\theta}(F) \vf_2 \big|\vf_1\big\rangle:=|\theta|_\k\,q^n
 \big\langle F\big|W_{\overline\varphi_1, \, \varphi_2}^{\theta}\big\rangle,\quad F\in\CD'(X_n)\quad
\mbox{and}\quad \vf_1,\vf_2\in\CD(U_n).
 $$
 \end{prop}
 
 In particular, we can give a meaning of $\bO_\theta(\mathds1_{X_n})$ ($\mathds1_{X_n}$ is the characteristic function of $X_n$) as a continuous operator from $\CD(U_n)$ to $\CD'(U_n)$. By 
definition  we have for all $\vf_1,\vf_2\in\CD(X_n)$:
\begin{align*}
 \big\langle{\bf \Omega}_{\theta}(\mathds 1_{X_n}) \vf_2 \big|\vf_1\big\rangle
&=q^n|\theta|_\k\int_{X_n}W_{\overline\vf_1,\vf_2}^\theta([g])\,d[g]\\
&=q^n|\theta|_\k\int_{U_n}\Big(\sum_{[t]\in\Gamma_n}W_{\overline\vf_1,\vf_2}^\theta(u,[t])\Big)du
=|\theta|_\k\int_{U_n}\Big(\int_\k W_{\overline\vf_1,\vf_2}^\theta(u,[t])\,dt\Big)du,
\end{align*}
where all the integrals are absolutely convergent by Lemma \ref{wignerbruhat}
and where we view $[t]\mapsto W_{\overline\vf_1,\vf_2}^\theta(u,[t])$ as a function on $\k$ invariant by translations in $\omega^{-n}\CO_\k$.
By \eqref{wig-four} we know that
$W_{\overline \vf_1,\vf_2}^\theta(u,[t])= 
\CF_{\k}^{-1}( f_{u})(t)$, where 
$$
f_{u} :=\Big[ z_0 \mapsto  |\theta|_{\k}^{-1}  \, \mathds{1}_{\theta\varpi^n\CO_{\k}}(z_0)  \,  
{\varphi_1}\big(\phi^{-1}(\theta^{-1}z_0) u\big) \, 
\varphi_2\big(u(\phi^{-1}(\theta^{-1}z_0))^{-1}\big)\Big]
,\quad u\in U_n.
$$
Hence we get
\begin{align*}
 \big\langle{\bf \Omega}_{\theta}(\mathds 1_{X_n}) \vf_2 \big|\vf_1\big\rangle&= |\theta|_{\k}\int_{U_n}  \CF_\k\CF_{\k}^{-1}( f_{u})(0)\,du
=|\theta|_{\k}\int_{U_n}f_u(0)\,du=\int_{U_n}{\vf_1}(u)\,\vf_2(u)\,du.
\end{align*}
Therefore, $\bO_\theta(\mathds1_{X_n})$ is just the natural injection $\CD(U_n)\hookrightarrow\CD'(U_n)$. (Hence $\bO_\theta(\mathds1_{X_n})={\rm Id}$ as bounded
operators
 on $L^2(U_n)$.)

Using \cite[Corollary 2 page 56]{Bruhat},
 it makes sense to talk about the Schwartz kernel of the operator ${\bf \Omega}_{\theta}
 (F)$, for $F$ a distribution:
  
\begin{lem}
\label{propOm}  
Let $\theta\in\CO_\k\setminus\{0\}$
and $F \in \CD'(X_n)$. Then  the Schwartz kernel  $\big[{\bf \Omega}_{\theta}(F)\big]$ of the continuous linear operator
${\bf \Omega}_{\theta}(F):\CD(U_n)\to\CD'(U_n)$ is  (with a little abuse of notation
and identifying  $\widehat\Gamma_n$ with $\varpi^n\CO_\k$ as usual) given by:
 \begin{equation}
 \label{KER}
\big[{\bf \Omega}_{\theta}(F)\big] (u_0,u) =    |\theta|_\k\,q^n \,  \big(\Id \otimes \CF_{\Gamma_n} F\big)
 \Big(u^{1/2}u_0^{1/2},
 \theta \phi\big( u^{1/2}u_0^{-1/2} \big)\Big),  \quad u,u_0 \in U_n.
 \end{equation}
\end{lem}

\begin{proof}
By continuity of the map $\CD'(X_n)\to \CL(\CD(U_n),\CD'(U_n))$, $F\mapsto  {\bf \Omega}_{\theta}(F) $, and of the map $\CD'(X_n)\to \CD'(U_n\times U_n)$
  which assigns to $F$
the RHS of \eqref{KER}  and
by density of $\CD(X_n)$ in $\CD'(X_n)$, it suffices to prove the formula  for symbols in $\CD(X_n)$.
So, let $F \in  \CD(X_n)$ and  $\varphi_1,\vf_2 \in \CD(U_n)$. Then, we have by  \eqref{wig1} and by Fubini's Theorem:
\begin{align*}
\big\langle{\bf \Omega}_{\theta}(F) \vf_2 \big|\overline \vf_1\big\rangle&=
|\theta|_\k\,q^n \big\langle F\big|W_{\varphi_1, \, \varphi_2}^{\theta}\big\rangle\\
& =|\theta|_\k\,q^n\sum_{[t]\in\Gamma_n}\int_{U_n} F(u,[t])\Big(
  \int_{U_n}  \overline{\varphi_1}(u_0)  \, \Psi_{\theta}\big(\phi(uu_0^{-1} ) t\big)  
   \,   \varphi_2(u^2u_0^{-1}) \, du_0\Big)du\\
   &=|\theta|_\k\,q^n \int_{U_n\times U_n} \big(\Id\otimes \CF_{\Gamma_n} F\big)
   \big(u,\theta \phi(uu_0^{-1})\big)  \,\overline{\varphi_1}(u_0)\,  \varphi_2(u^2u_0^{-1}) \, du_0\,du.
\end{align*}
Now, the (first) substitution formula \eqref{change} followed by the translation $u \mapsto uu_0$
gives
\begin{align*}
\big\langle{\bf \Omega}_{\theta}(F) \vf_2 \big|\overline \vf_1\big\rangle&=
|\theta|_\k\,q^n  \int_{U_n\times U_n} \big(\Id\otimes \CF_{\Gamma_n} F\big)
   \big(u^{1/2},\theta \phi(u^{1/2}u_0^{-1})\big)  \,\overline{\varphi_1}(u_0)\,  \varphi_2(uu_0^{-1}) \, 
   du_0\,du\\
   &=
|\theta|_\k\,q^n \int_{U_n\times U_n} \big(\Id\otimes \CF_{\Gamma_n} F\big)
   \big(u^{1/2}u_0^{1/2},\theta \phi(u^{1/2}u_0^{-1/2})\big)  \,\overline{\varphi_1}(u_0)\,  \varphi_2(u) \, 
   du_0\,du,
\end{align*}
which completes the proof.
\end{proof}
\begin{rmk}
\label{INJ}
{\rm 
Provided $\theta \in  \CO_{\k}^{\times}$, Lemma \ref{propOm} also shows that the map $\CD'(X_n)\to \CD'(U_n\times U_n)$, which assigns to a symbol $F$
the Schwartz kernel of the linear operator $\bO_\theta(F)$,  is an injection. This proves that the quantization map ${\bf \Omega}_{\theta}:\CD'(X_n)\to
\mathcal L\big(\CD(U_n),\CD'(U_n)\big)$ is injective when $\theta \in  \CO_{\k}^{\times}$.
}
\end{rmk}

From Lemma \ref{propOm}, we can easily determine some particular classes of pseudo-differential operators in the $p$-adic Fuchs calculus.
\begin{cor}
Let $\theta\in\CO_\k\setminus\{0\}$.
If $F\in\CD'(X_n)$ is invariant under  left translations in $\Gamma_n$ (respectively in $U_n$), then ${\bf \Omega}_{\theta}(F)$
is a multiplication (respectively convolution) operator.
\end{cor}
\begin{proof}
Assume  first that $F\in\CD'(X_n)$ is invariant under  left translations in $\Gamma_n$. This implies (with usual abuse of notation) that $F(u,[t])=T(u)$
for some $T\in\CD'(U_n)$. Hence, with $\delta_t$ the Dirac measure
on $\k$ supported at $t\in\k$, the Schwartz kernel of the operator  ${\bf \Omega}_{\theta}(F)$ is given by
$$
 |\theta|_\k\
T \big(u^{1/2}u_0^{1/2}\big)\,
 \delta_0\big(\theta \phi( u^{1/2}u_0^{-1/2}) \big).
 $$
 Since $|\phi'(u)|_\k=|\sigma'(u)|_\k=1$ by Proposition \ref{sigma}, the  Schwartz kernel of  ${\bf \Omega}_{\theta}(F)$ is therefore given by  $T(u_0)\, \delta_{u_0}\big(u)$.
 Hence, ${\bf \Omega}_{\theta}(F):\CD(U_n)\to\CD'(U_n)$ is the operator of multiplication by $T\in\CD'(U_n)$.
 
 Assume then that $F\in\CD'(X_n)$ is invariant under  left translations in $U_n$. This implies (with usual abuse of notation) that $F(u,[t])=T([t])$
for some $T\in\CD'(\Gamma_n)$. Hence the Schwartz kernel of  ${\bf \Omega}_{\theta}(F)$ is given by
$$
 |\theta|_\k\,q^n \,   \CF_{\Gamma_n} T \big(\theta \phi( u^{1/2}u_0^{-1/2}) \big).
$$
This shows that ${\bf \Omega}_{\theta}(F):\CD(U_n)\to\CD'(U_n)$ is the operator of convolution by the distribution of $U_n$ given by
$\big[u\mapsto |\theta|_\k\,q^n \,   \CF_{\Gamma_n} T \big(\theta \phi( u^{-1/2})\big)\big]$.
\end{proof}

The next result is fundamental for us since it will allow to invert the quantization map $\bO_\theta$.
 For  the invertibility  property to hold true, we need to restrict the parameter space to be the  multiplicative
    group of  units $\CO_\k^\times$. 
 
 In what follows, $\CL^2(L^2(U_n))$ denotes the Hilbert space of Hilbert-Schmidt operators on 
 $L^2(U_n)$. More generally, we denote by $\CL^p(L^2(U_n))$, $p\in[1,\infty)$, 
 the $p$-th Schatten class of  operators on $L^2(U_n)$.

\begin{prop}
\label{Omegaunitar} 
Let $\theta \in \CO_{\k}^{\times}$. Then the map $ q^{-n/2}{\bf \Omega}_{\theta}:\CD'(X_n)\to
\mathcal L\big(\CD(U_n),\CD'(U_n)\big)$, restricts to 
a surjective isometry from   $L^2(X_n)$ to $\CL^2( L^2(U_n) )  $.
 \end{prop}

\begin{proof}
We first show that, $q^{-n/2}{\bf \Omega}_{\theta}$ is  an isometry and for that issue, we compute the 
Hilbert-Schmidt norm of $ {\bf \Omega}_{\theta}(f)$, $f\in L^2(X_n)$, as the $L^2$-norm of its kernel
$\big[{\bf \Omega}_{\theta}(f)\big] $.
Using the conclusion of Lemma \ref{propOm} together with the invariance of the Haar measure of $U_n$
by dilations in $\CO_\k^\times$ and group inversion,  with the substitution formulas \eqref{change} and with the Plancherel
formula for $\Gamma_n$,
we get since $\theta\in\CO_\k^\times$:
\begin{align*}
\int_{U_n^2}    \big| \big[{\bf \Omega}_{\theta}(f)\big] (u_0,u)    \big|^2  \, du du_0 &=  
q^{2n} \, \int_{U_n\times U_n}   \big |         \big(\Id\otimes \CF_{\Gamma_n} f\big)
\big(u^{1/2}u_0^{1/2},\theta \phi( u^{1/2}u_0^{-1/2} ) \big)    \big |^2  \, du  du_0 \\
&=  
q^{2n} \, \int_{U_n\times U_n}   \big |         \big(\Id\otimes \CF_{\Gamma_n} f\big)
\big(u^{1/2},\theta \phi( u^{1/2}u_0^{-1} ) \big)    \big |^2  \, du  du_0 \\
&=  
q^{2n} \, \int_{U_n\times U_n}   \big |         \big(\Id\otimes \CF_{\Gamma_n} f\big)
\big(u,\theta \phi( uu_0^{-1} ) \big)    \big |^2  \, du  du_0 \\
&=  
q^{2n} \, \int_{U_n\times U_n}   \big |         \big(\Id\otimes \CF_{\Gamma_n} f\big)
\big(u,\theta \phi( u_0 ) \big)    \big |^2  \, du  du_0 \\
&=  
q^{2n} \, \int_{U_n\times \varpi^n\CO_\k}   \big |         \big(\Id\otimes \CF_{\Gamma_n} f\big)
\big(u, z_0  \big)    \big |^2  \, du  dz_0 \\
&=  q^n
\sum_{[t]\in\Gamma_n}\int_{U_n}   \big |   f
(u, [t])   \big |^2  \, du  =q^n\int_{X_n} \big |f([g]) \big |^2\,d[g].
 \end{align*}

Consider last the linear operator which assigns to a symbol $f$ the kernel of the operator  ${\bf \Omega}_{\theta}(f)$:
$$
L^2(X_n)\to L^2(U_n\times U_n),\quad
f\mapsto \Big[(u_0,u)\mapsto   \big(\Id\otimes \CF_{\Gamma_n} f\big)
\big(u^{1/2}u_0^{1/2},\theta \phi( u^{1/2}u_0^{-1/2} ) \big)  \Big].
$$
This operator is unitary  since  the map 
$$
U_n\times U_n\to U_n\times U_n,\quad (u_0,u)\mapsto \big(u^{1/2}u_0^{1/2},\theta \phi( u^{1/2}u_0^{-1/2} ) \big),
$$
is clearly an homeomorphism of class $C^1$ with Jacobian $1$. Therefore, 
the quantization map is also surjective. 
 \end{proof}
 
 \begin{rmks}
 \label{rm}
 {\rm  {\it (i)} The adjoint $\bO^*_\theta :\CL^2(L^2(U_n))\to L^2(X_n)$ of the quantization map (that is, the symbol map) is
  given for a trace-class operator $A\in \CL^1(L^2(U_n))$ by the function $\big[[g]\mapsto q^n{\rm Tr}\big(A\Omega_\theta([g])\big)\big]\in L^2(X_n)\cap L^\infty(X_n)$.
  Indeed, since  $\|\Omega_\theta( [g])\|=1$, we can use Fubini's Theorem (for the von Neumann algebra $\mathcal B(L^2(U_n))\bar\otimes L^\infty(X_n)$ with semi-finite trace ${\rm Tr}\otimes\int_{X_n}$) to get   for $f\in L^1(X_n)\cap L^2(X_n)$ and $A\in \CL^1(L^2(U_n))$:
  $$
  \langle {\bf \Omega}_{\theta}^*(A),f\rangle ={\rm Tr}\big(A^*{\bf \Omega}_{\theta}(f)\big)=
  q^n{\rm Tr}\Big(A^*\int_{X_n}\Omega_{\theta}([g])\,f([g])d[g]\Big)
  = q^n\int_{X_n} \overline{{\rm Tr}\big(A\Omega_{\theta}([g])\big)}\,f([g])d[g].
   $$
{\it (ii)} Unitarity of the quantization map also implies that ${\bf \Omega}_{\theta}(W_{\varphi_{1},  \varphi_{2}}^{\theta}) = | \varphi_{2 }\rangle \langle \varphi_1|$.
Indeed, using Dirac's ket-bra notation
for rank one operators:
$$ 
| \varphi_{2 }\rangle \langle \varphi_{1} | (\vf )  :=\langle\vf_1,\vf\rangle\, \varphi_2,
\quad \vf,\vf_1,\vf_2\in L^2(U_n),
$$
we get
$$
W_{\varphi_{1},  \varphi_{2}}^{\theta}([g])=\langle\vf_1,\Omega_\theta([g])\vf_2\rangle ={\rm Tr}\big(  | \varphi_{2 }\rangle \langle \varphi_1| \,\Omega_\theta([g])\big)
=q^{-n} \bO^*_\theta\big( | \varphi_{2 }\rangle \langle \varphi_1| \big)
=\bO^{-1}_\theta\big( | \varphi_{2 }\rangle \langle \varphi_1| \big).
$$
}
 \end{rmks}
 
 \subsection{The composition law of symbols in $\CD(X_n)$}
 \label{CLSL2}

 We shall now give  the formula for the composition  law of symbols   in the $p$-adic Fuchs calculus
 and prove the basic properties of this non-formal $\star$-product. The latter is 
   defined by transporting the product of the Hilbert-Schmidt operators $\CL^2(L^2(U_n))$
  to $L^2(X_n)$ via the unitary operator $q^{-n/2}\bO_\theta$:
 \begin{dfn}
 Let $\theta\in\CO_\k^\times$. We let $\star_\theta$ be the associative (continuous) product on 
 $L^2(X_n)$ given by:
 $$
 f_1\star_\theta f_2:=\bO_\theta^{-1}\big(\bO_\theta(f_1)\,\bO_\theta(f_2)\big).
 $$
 \end{dfn}
 
 \begin{rmk}
 {\rm
 Since $\bO_\theta(f)^*=\bO_\theta(\overline f)$, we see that the complex conjugation is an    
 involution of the Hilbert algebra $(L^2(X_n),\star_\theta)$.}
 \end{rmk}
 It is easy to give an explicit expression for the deformed product $\star_\theta$ between test functions:
 \begin{prop}
 \label{IF}
 Let $\theta\in\CO_\k^\times$. For $f_1,f_2\in\CD(X_n)$, we have
 $$
 f_1\star_\theta f_2([g])=\int_{X_n\times X_n} K_{\theta}^3\big([g],[g_1],[g_2]\big)\,
 f_1([g_1])\,f_2([g_2])\,d[g_1]\,d[g_2],
 $$
 where $K_\theta^3$ is the locally constant on $X_n^3$ given by:
 $$
 K_{\theta}^3 \big ( (u_1,[t_1]),(u_2,[t_2]),(u_3,[t_3]) \big ) := q^{2 n } \, \Psi_{\theta}\big(
  \phi(u_1 u_2^{-1})t_3+ \phi( u_2 u_3^{-1})t_1+\phi( u_3 u_1^{-1})t_2\big).
 $$
 \end{prop}
 \begin{proof}
 For $f_1,f_2\in \CD(X_n)$, $\bO_\theta(f_j)$ is Hilbert-Schmidt and thus
 $\bO_\theta(f_1)\,\bO_\theta(f_2)$ is trace class.
Hence, Remark \ref{rm} (i)
 entails that:
 $$
 f_1\star_\theta f_2([g])=
 {\rm Tr}\big(\bO_\theta(f_1)\,\bO_\theta(f_2)\,\Omega_\theta([g])\big).
 $$
 Set 
$A_1:= \bO_\theta(f_1)$ and $A_2:=\bO_\theta(f_2)\,\Omega_\theta([g])$. Being Hilbert-Schmidt, the kernels  of $A_1$ and $A_2$
(denoted by $[A_1], [A_ 2]$ in what follows) belong to $L^2(U_n\times U_n)$.
Polarising the equality between the Hilbert-Schmidt norm of an operator and the $L^2$-norm   of its kernel, we deduce 
$$
{\rm Tr}(A_1A_2)=\int_{U_n\times U_n} [A_1](u_1,u_2)\, [A_ 2](u_2,u_1)\,du_1 du_2.
$$  
By Lemma  \eqref{propOm}
$$
[A_1](u_1,u_2)= q^n \,  \big(\Id \otimes \CF_{\Gamma_n} f_1\big)
 \Big(u_2^{1/2}u_1^{1/2},
 \theta \phi\big( u_2^{1/2}u_1^{-1/2} \big)\Big).
$$
By  Lemma \ref{omegexp}, we easily deduce that
$$
\big[\Omega_\theta([g])\big](u_3,u_1) = \Psi_{\theta}\big( \phi(uu_3^{-1}) t\big)    \delta_{u^2 u_3^{-1}}(u_1)=
\Psi_{\theta}\big( \phi(u_1u^{-1}) t\big)    \delta_{u^2 u_1^{-1}}(u_3),
$$
and thus
\begin{align*}
[A_ 2](u_2,u_1)&=\int_{U_n}\big[\bO_\theta(f_2)\big](u_2,u_3)\,\big[\Omega_\theta([g])\big](u_3,u_1)\,du_3\\
&=\big[\bO_\theta(f_2)\big](u_2,u^2 u_1^{-1})\,\Psi_{\theta}\big( \phi(u_1u^{-1}) t\big)\\
&=q^n\, \big(\Id \otimes \CF_{\Gamma_n} f_2\big) \Big(uu_1^{-1/2}u_2^{1/2},\theta \phi\big(uu_1^{-1/2}u_2^{-1/2} \big)\Big)\,
 \Psi_{\theta}\big( \phi(u_1u^{-1}) t\big).
\end{align*}
 Hence, we get
\begin{align*}
{\rm Tr}(A_1A_2)&=q^{2n}\int_{U_n\times U_n}\big(\Id \otimes \CF_{\Gamma_n} f_1\big)\Big(u_2^{1/2}u_1^{1/2}, \theta \phi\big( u_2^{1/2}u_1^{-1/2} \big)\Big)\\
&\qquad\qquad\qquad\times  \big(\Id \otimes \CF_{\Gamma_n} f_2\big) \Big(uu_1^{-1/2}u_2^{1/2},\theta \phi\big(uu_1^{-1/2}u_2^{-1/2} \big)\Big)\,
 \Psi_{\theta}\big( \phi(u_1u^{-1}) t\big)\,du_1du_2\\
 &
 \hspace{-1cm}=q^{2n}\int_{U_n\times U_n}\big(\Id \otimes \CF_{\Gamma_n} f_1\big)\big(u_1, \theta \phi\big( u_2u^{-1} \big)\big)
   \big(\Id \otimes \CF_{\Gamma_n} f_2\big) \big(u_2,\theta \phi\big(uu_1^{-1} \big)\big)\,
 \Psi_{\theta}\big( \phi(u_1u_2^{-1}) t\big)\,du_1du_2.
\end{align*}
Undoing the Fourier transforms, we arrive at the announced formula. 
  \end{proof}
 
 By construction, the quantization map $\bO_\theta$ intertwines the geometric action of $G_n$ 
 on $L^2(X_n)$ with the action of $G_n$ on $\CL^2(L^2(U_n))$ given by conjugation of the
 representation $\pi_{\theta}$. This implies that the product $\star_\theta$ on $L^2(X_n)$
 is covariant for the action of $G_n$ or equivalently, that the three-point functions $ K_\theta^3$
 in invariant under the diagonal action of $G_n$ (a property that can easily be  checked directly).
 Since $\theta$ is now assumed to be in $\CO_\k^\times$, the representation $\widetilde \pi_{\theta}$ of 
 the quotient group $X_n$ is no longer defined--see the discussion right after equation \eqref{Utheta}.
  Hence, it is not automatic that the product
 $\star_\theta$ is covariant for the left action of $X_n$ on $L^2(X_n)$ (a crucial property 
for deformation theory for actions of $X_n$ on $C^*$-algebras \cite{GJ3}).
In fact, the invariance of the function
$$
G_n\ni (u,t)\mapsto \Psi_\theta\big(\phi(u)t\big),
$$
under translations in $H_n=\{1\}\ltimes\varpi^{-n}\CO_\k$ 
immediately implies that  $ K_\theta^3$ is also invariant under
the diagonal left action of $X_n$. Hence, denoting by $\lambda$ the left regular action and by
$\rho$ the right regular action of $X_n$ on $L^2(X_n)$,  we get:
 
 \begin{prop}
 Let $\theta\in\CO_\k^\times$.
 The product $\star_\theta$ is covariant for the left action of $X_n$, namely
 $$
 \lambda_{[g]}\big(f_1\star_\theta f_2\big)= \lambda_{[g]}(f_1)\star_\theta \lambda_{[g]}(f_2),
 \quad \forall f_1,f_2\in L^2(X_n)\,,\;\forall [g]\in X_n.
 $$
 In particular, when $f_1,f_2\in\CD(X_n)$, we have
 $$
 f_1\star_\theta f_2=\int_{X_n\times X_n} K_{\theta}\big([g_1],[g_2]\big)\,
\rho_{[g_1]}( f_1)\,\rho_{[g_2]}(f_2)\,d[g_1]\,d[g_2],
 $$
where 
$$
K_{\theta}\big([g_1],[g_2]\big):=K^3_{\theta}\big([e],[g_1],[g_2]\big)
=q^{2n} \, \Psi_{\theta}\big(
  \phi( u_2 )t_1-\phi(u_1)t_2\big).
  $$
 \end{prop}
 
As in Lemma \ref{wignerbruhat},  we have to remember here
 that for Abelian $p$-adic groups, Bruhat's notions of smooth compactly supported function and of Schwartz function are identical.

 \begin{prop}
 \label{classe}
  Let $\theta\in\CO_\k^\times$.
 The product $\star_\theta$ is  continuous from $\CD(X_n)\times\CD(X_n)$ to
 $\CD(X_n)$.
 \end{prop}
\begin{proof}
For $f_1,f_2\in\CD(X_n)$, we have
 $$
 f_1\star_\theta f_2(u,[t])=q^{2 n }\int_{X_n\times X_n} \Psi_{\theta}\big(
  \phi(u u_1^{-1})t_2+ \phi( u_1 u_2^{-1})t+\phi( u_2 u^{-1})t_1\big)\,
 f_1([g_1])\,f_2([g_2])\,d[g_1]\,d[g_2].
 $$
 Since $\Psi_\theta$ is locally constant, we deduce from the continuity of $\phi$, from the fact that $u_1, u_2$
 are units and from the fact that the sum over $[t_1]$ and $[t_2]$ is actually finite, that $ f_1\star_\theta f_2$
 is locally constant in the variable $u\in U_n$. To show that $ f_1\star_\theta f_2\in\CD(X_n)$, we therefore need to show that $ f_1\star_\theta f_2$  is finitely supported in the variable
 $[t]\in\Gamma_n$. For that, we perform the sums over $[t_1],[t_2]\in\Gamma_n$ to get
  $$
 f_1\star_\theta f_2(u,[t])=q^{2 n }\!\!\int_{U_n\times U_n}\!\!\!\!\!\! \Psi_{\theta}\big(
  \phi( u_1 u_2^{-1})t\big)\,
 \big(\Id\otimes \CF_{\Gamma_n} f_1\big)(u_1,\phi( u_2 u^{-1}))\, \big(\Id\otimes \CF_{\Gamma_n} f_2\big)(u_2,\phi(u u_1^{-1}))\,du_1\,du_2.
 $$
 Using now the translation $u_1\mapsto u_1u_2$ and the second substitution formula \eqref{change}, we get
  \begin{align*}
 &f_1\star_\theta f_2(u,[t])=q^{2 n }\int_{\varpi^n\CO_\k} \Psi_{\theta}(
   u_1 t)\\
  & \qquad\times\Big(\int_{U_n}
 \big(\Id\otimes \CF_{\Gamma_n} f_1\big)\big(\phi^{-1}(u_1)u_2,\phi( u_2 u^{-1})\big)\, \big(\Id\otimes \CF_{\Gamma_n} f_2\big)\big(u_2,\phi(u u_2^{-1}
 \phi^{-1}(u_1)^{-1})\big)\,du_2\Big)\,du_1.
 \end{align*}
 Hence, 
 $ f_1\star_\theta f_2(u,[t])=q^{2 n }\CF_\k(g_u)(\theta t)$
 where
 $$
 g_u(t)=\mathds{1}_{\varpi^n\CO_\k}(t)\int_{U_n}
 \big(\Id\otimes \CF_{\Gamma_n} f_1\big)\big(\phi^{-1}(t)u_2,\phi( u_2 u^{-1})\big)\, \big(\Id\otimes \CF_{\Gamma_n} f_2\big)\big(u_2,\phi(u u_2^{-1}
 \phi^{-1}(t)^{-1})\big)\,du_2.
 $$
 Since $f_1$ is locally constant in the variable $u\in U_n$, we deduce that  $\Id\otimes \CF_{\Gamma_n} f_1$ is locally constant in its first variable.
 Similarly, $f_2$ is finitely supported in the variable $[t]\in\Gamma_n$ so that $\Id\otimes \CF_{\Gamma_n} f_2$ is locally constant in its 
 second variable. From this we deduce that $g_u$ is locally constant, so that $\CF_\k(g_u)$ is compactly supported. Hence, 
 $ f_1\star_\theta f_2$ is finitely supported in the variable $[t]\in\Gamma_n$. Continuity follows from the obvious estimate
 $$
 \|f_1\star_\theta f_2\|_\infty\leq q^{2 n }\, {\rm Vol}({\rm Supp}(f_1))\, {\rm Vol}({\rm Supp}(f_2))\,\|f_1\|_\infty\,\|f_2\|_\infty,
 $$
 and from the fact that 
 the invariance domain and the support of
 $ f_1\star_\theta f_2$ both depend only on  the invariance domains and the supports of $f_1,f_2$ but not on $f_1$ and $f_2$ themselves.  
  \end{proof}
  
  The next property  is often named \emph{strong traciality}:
 
 \begin{prop}
  Let $\theta\in\CO_\k^\times$.
 For $f_1,f_2\in\CD(X_n)$, we have
 $$
 \int_{X_n}  f_1\star_\theta f_2([g])\,d[g]= \int_{X_n}  f_1([g])\,f_2([g])\,d[g].
 $$
 \end{prop}
 \begin{proof}
 Observe first that the left hand side is well defined. Indeed, Proposition \ref{classe} shows that $f_1\star_\theta f_2$ belongs to $\CD(X_n)\subset L^1(X_n)$.

By polarization, we may assume $f_1=\bar f_2$. 
 So, let $f\in\CD(X_n)$.
 By unitarity of the $p$-adic Fuchs calculus (see Proposition \ref{Omegaunitar}), we have since $\CD(X_n)\subset L^2(X_n)$:
 $$
\int_{X_n}  \bar f([g])\,f([g])\,d[g]=\|f\|_2^2=q^{-n}\,\|\bO_\theta( f)\|_2^2.
 $$
From Lemma \ref{propOm}, the operator kernel of $\bO_\theta( \bar f\star_\theta f)$ belongs to $\CD(U_n\times U_n)$.
Hence, its evaluation  on the diagonal  is well defined as an element of $\CD(U_n)$. Moreover, we have the following formula:
 $$
 \big[{\bf \Omega}_{\theta}(\bar f\star_\theta f)\big] (u,u) =   q^n \,  \big(\Id \otimes \CF_{\Gamma_n} ( \bar f\star_\theta f)\big)
(u,0)=   q^n \,\sum_{[t]\in\Gamma_n} \bar f\star_\theta f(u,[t]).
 $$
 Since $\bar f\star_\theta f \in\CD(X_n)\subset L^1(X_n)$, we deduce by Fubini's Theorem that the map $u\mapsto  \big[{\bf \Omega}_{\theta}(\bar  f\star_\theta f)\big] (u,u)$
 belongs to $L^1(U_n)$ with
 $$
 \int_{U_n}\big[{\bf \Omega}_{\theta}(\bar f\star_\theta f)\big] (u,u)\,du=q^n\int_{X_n}\bar f\star_\theta f([g])\,d[g].
 $$
Last, we observe that the product kernel formula gives:
 \begin{align*}
 \int_{U_n}\big[{\bf \Omega}_{\theta}(\bar f\star_\theta f)\big] (u,u)\,du&= \int_{U_n\times U_n}\big[{\bf \Omega}_{\theta}(f)^*\big] (u,u_1)
 \big[{\bf \Omega}_{\theta}( f)\big] (u_1,u)\,dudu_1\\
&= \int_{U_n\times U_n}\overline{\big[{\bf \Omega}_{\theta}(f)\big]} (u_1,u)
 \big[{\bf \Omega}_{\theta}(f)\big] (u_1,u)\,dudu_1
 =\big\|\big[{\bf \Omega}_{\theta}(f)\big]\big\|_2^2=\|\bO_\theta(\overline f)\|_2^2,
 \end{align*}
 which completes the proof.
 \end{proof}
 \section{A Calder\'on-Vaillancourt  type Theorem}
 \label{CV}
\subsection{The symbol space}

The aim of this subsection is to construct a Fr\'echet space of functions on
$X_n$ suitable to extend the Calder\'on-Vaillancourt Theorem  to the $p$-adic Fuchs calculus.

We will make use of the following continuous function on $\k$:
\begin{equation}
\mu_0(t):=\max(1, |\varpi^n t|_{\bf k}),\quad t\in\k .
\end{equation}
It is known (see for instance \cite{Haran}) that $\mu_0^{-1-\eps} \in L^1(\k)$ for every $\eps>0$ and
that $\mu_0$ satisfies a Peetre type inequality:
\begin{align}
\label{Peetre}
 \mu_0(t_1+t_2)\leq\mu_0(t_1)\,\mu_0(t_2),\quad\forall t_1,t_2\in\k.
 \end{align}
Of course,   $\mu_0$ is invariant under  dilations in $\CO_\k^\times$ and under translations in $\varpi^{-n}\CO_\k$. 
 The following properties of the distribution 
 $ \CF_{{\bf k}}\big(\overline{\Psi}\mu_0^s\big)$, $s\in\R$, are important. 

\begin{lem}
\label{corsigmaphi}
Let $s \in\R$. Then, the distribution $ \CF_{{\bf k}}\big(\overline{\Psi}\mu_0^{s}\big) 
 \in\CD'(\k)$ is real valued, has support 
contained in $U_n$ and is invariant under  multiplicative inversion.
\end{lem}

\begin{proof}
That $ \CF_{{\bf k}}\big(\overline{\Psi}\mu_0^{s}\big) $ is real, is obvious since 
 $\mu_0$ is real and even and $\Psi$ is an additive character.
That the support of $ \CF_{{\bf k}}\big(\overline{\Psi}\mu_0^{s}\big)$ is contained in $U_n$ 
comes from the invariance of  $\mu_0$  by translations in  $\varpi^{-n}\CO_{\k}$.
In particular,  $ \CF_{{\bf k}}\big(\overline{\Psi}\mu_0^{s}\big)$  can be viewed 
as an element of $\CD'(U_n)$.

Let $S$ be the involutive homeomorphism of 
$\CD(U_n)$ (and of $\CD'(U_n)$) given by $S\vf(u):=\vf(u^{-1})$. 
(Recall that we denote by $\langle T|\vf\rangle$ the evaluation of a 
distribution $T$ on a test function $\vf$.)
For $\vf\in\CD(U_n)$ (that we also view
as an element of $\CD(\k)$ supported on $U_n$), we have
\begin{align*}
\big\langle S\, \CF_{{\bf k}}\big(\overline{\Psi}\mu_0^{s}\big)\big|\vf\big\rangle&=
 \big\langle\CF_{{\bf k}}\big(\overline{\Psi}\mu_0^{s}\big)\big|S\vf\big\rangle=
\big\langle \overline{\Psi}\mu_0^{s}\big|\CF_{{\bf k}}S\vf\big\rangle
=\int_\k \overline{\Psi}(t) \, \mu_0^{s}(t)\,
\big(\CF_{{\bf k}}\, S\vf\big)(t)\,dt,
\end{align*}
and by dominated convergence we get
$$
\big\langle S\, \CF_{{\bf k}}\big(\overline{\Psi}\mu_0^{s}\big)\big|\vf\big\rangle=
\lim_{\ell\to+\infty} \int_{\varpi^{-\ell}\CO_\k} \overline{\Psi}(t) \, \mu_0^{s}(t)\,
\big(\CF_{{\bf k}}\,S\vf\big)(t)\,dt.
$$
Hence we are  left with integrals over compact sets, so that we may use Fubini's Theorem  to write
\begin{align*}
\int_{\varpi^{-\ell}\CO_\k} \overline{\Psi}(t) \, \mu_0^{s}(t)\,
\big(\CF_{{\bf k}}\, S\vf\big)(t)\,dt&=
\int_{\varpi^{-\ell}\CO_\k\times U_n} \overline{\Psi}(t) \, \mu_0^{s}(t)\,
\Psi(ut)\,\vf(u^{-1})\,dtdu\\
&=
\int_{\varpi^{-\ell}\CO_\k\times U_n} \mu_0^{s}(t)\,
\Psi\big((u^{-1}-1)t\big)\,\vf(u)\,dtdu\\
&=
\int_{\varpi^{-\ell}\CO_\k\times U_n} \mu_0^{s}(t)\,
\Psi\big(-u^{-1}(u-1)t\big)\,\vf(u)\,dtdu.
\end{align*}
Since $\mu_0$ is also invariant under  dilations in the group of units, the change of variable
$t\mapsto- ut$ yields
\begin{align*}
\int_{\varpi^{-\ell}\CO_\k} \overline{\Psi}(t) \, \mu_0^{s}(t)\,
\big(\CF_{{\bf k}} S\vf\big)(t)\,dt&=
\int_{\varpi^{-\ell}\CO_\k\times U_n} \mu_0^{s}(t)\,
\Psi\big((u-1)t\big)\,\vf(u)\,dtdu\\
&=\int_{\varpi^{-\ell}\CO_\k} \overline{\Psi}(t) \, \mu_0^{s}(t)\,
\big(\CF_{{\bf k}}\vf\big)(t)\,dt,
\end{align*}
which from Fubini's Theorem  and dominated convergence used backward yields the result.
\end{proof}

\begin{rmk}
{\rm In what follows, we will mostly view  
$ \CF_{{\bf k}}\big(\overline{\Psi}\mu_0^{s}\big)$ as an  element of $\CD'(U_n)$.}
\end{rmk}

We now come to a key technical result.
Denote by  $\ast_{U_n}$ the convolution product on $U_n$. Since $U_n$ is a compact group, $\ast_{U_n}$
extends to an associative bilinear continuous mapping from $\CD'(U_n)\times \CD'(U_n)$ to
$\CD'(U_n)$ (see \cite[section 6]{Bruhat}). The following result shows that  the multiplicative
convolution product
of the distributions $\CF_{{\bf k}}\big(\overline{\Psi}\mu_0^{s}\big)\in\CD'(U_n)$, $s\in\R$,
 behaves (almost) like the additive convolution product!

\begin{lem}
\label{strange}
For $s_1,s_2\in\R$, we have 
$$
\CF_{{\bf k}}\big(\overline{\Psi}\mu_0^{s_1}\big)\ast_{U_n}
\CF_{{\bf k}}\big(\overline{\Psi}\mu_0^{s_2}\big)=
\CF_{{\bf k}}\big(\overline{\Psi}\mu_0^{s_1+s_2}\big).
$$
\end{lem}
\begin{proof}
Assume first $s_1\in\R$ and $s_2<-1/2$. In this case, $\mu_0^{s_2}\in L^2(\k)$ and, since  $\CF_\k(\overline\Psi\mu_0^{s_2})$ is supported in $U_n$,
$\CF_\k(\overline\Psi\mu_0^{s_2})\in L^1(\k)\cap L^2(\k)$. Take $\vf\in\CD(U_n)$. Then,
we have by definition of the convolution product of a pair of distributions:
$$
\big\langle\CF_{{\bf k}}\big(\overline{\Psi}\mu_0^{s_1}\big)\ast_{U_n}
\CF_{{\bf k}}\big(\overline{\Psi}\mu_0^{s_2}\big)\big|\vf\big\rangle=
\big\langle\CF_{{\bf k}}\big(\overline{\Psi}\mu_0^{s_1}\big)\otimes
\CF_{{\bf k}}\big(\overline{\Psi}\mu_0^{s_2}\big)\big|\Delta\vf\big\rangle,
$$
where $\Delta:\CD(U_n)\to\CD(U_n\times U_n)$ is the coproduct, that is
$\Delta\vf(u_1,u_2)=\vf(u_1u_2)$. If we view  $\Delta\vf$ as an element of
$\CD(\k\times U_n)$ supported in $U_n\times U_n$, we get 
$$
\big\langle\CF_{{\bf k}}\big(\overline{\Psi}\mu_0^{s_1}\big)\ast_{U_n}
\CF_{{\bf k}}\big(\overline{\Psi}\mu_0^{s_2}\big)\big|\vf\big\rangle=
\big\langle\big(\overline{\Psi}\mu_0^{s_1}\big)\otimes
\CF_{{\bf k}}\big(\overline{\Psi}\mu_0^{s_2}\big)\big|\CF_\k\otimes{\rm Id}(\Delta\vf)\big\rangle,
$$
where the expression in the right hand side above is the evaluation of a distribution 
in $\CD'(\k\times U_n)$ against a test function in $\CD(\k\times U_n)$.
Since $\big(\overline{\Psi}\mu_0^{s_1}\big)$ is continuous  and 
$\CF_{{\bf k}}\big(\overline{\Psi}\mu_0^{s_2}\big)$ belongs to $L^1\cap L^2(U_n)$, we  have
\begin{align*}
\big\langle\CF_{{\bf k}}\big(\overline{\Psi}\mu_0^{s_1}\big)\ast_{U_n}
\CF_{{\bf k}}\big(\overline{\Psi}\mu_0^{s_2}\big)\big|\vf\big\rangle&=\int_{\k\times U_n}
\overline{\Psi}(t_1) \, \mu_0^{s_1}(t_1)\,
\CF_{{\bf k}}\big(\overline{\Psi}\mu_0^{s_2}\big)(u_2)\,\CF_\k\otimes{\rm Id}(\Delta\vf)
(t_1,u_2)\, dt_1\,du_2\\
&=\int_{\k\times U_n}
\overline{\Psi}(t_1) \, \mu_0^{s_1}(t_1)\,
\CF_{{\bf k}}\big(\overline{\Psi}\mu_0^{s_2}\big)(u_2)\Big(\int_{U_n}\Psi(t_1u_1)\,\vf(u_1u_2)\,du_1
\Big)dt_1\,du_2.
\end{align*}
Using dominated convergence, Fubini's Theorem, the change of variable
$u_1\mapsto u_2^{-1}u_1$, the invariance of 
$\CF_{{\bf k}}\big(\overline{\Psi}\mu_0^{s_2}\big)$  under group inversion 
(see Lemma \ref{corsigmaphi}) and the fact that $\vf$ can be seen as an element of $\CD(\k)$
supported on $U_n$, the above expression is equal to:
\begin{align*}
&\lim_{\ell\to+\infty}\int_{\varpi^{-\ell}\CO_\k\times U_n}
\overline{\Psi}(t_1) \, \mu_0^{s_1}(t_1)\,
\CF_{{\bf k}}\big(\overline{\Psi}\mu_0^{s_2}\big)(u_2)\Big(\int_{U_n}\Psi(t_1u_1)\,\vf(u_1u_2)\,du_1
\Big)dt_1\,du_2\\
&=\lim_{\ell\to+\infty}\int_{U_n\times\varpi^{-\ell}\CO_\k\times U_n}
\overline{\Psi}(t_1) \, \mu_0^{s_1}(t_1)\,
\CF_{{\bf k}}\big(\overline{\Psi}\mu_0^{s_2}\big)(u_2)\,\Psi(t_1u_1)\,\vf(u_1u_2)\,du_1
\,dt_1\,du_2\\
&=\lim_{\ell\to+\infty}\int_{U_n\times\varpi^{-\ell}\CO_\k\times U_n}
\overline{\Psi}(t_1) \, \mu_0^{s_1}(t_1)\,
\CF_{{\bf k}}\big(\overline{\Psi}\mu_0^{s_2}\big)(u_2)\,\Psi(t_1u_2^{-1}u_1)\,\vf(u_1)\,du_1
\,dt_1\,du_2\\
&=\lim_{\ell\to+\infty}\int_{U_n\times\varpi^{-\ell}\CO_\k\times U_n}
\overline{\Psi}(t_1) \, \mu_0^{s_1}(t_1)\,
\CF_{{\bf k}}\big(\overline{\Psi}\mu_0^{s_2}\big)(u_2)\,\Psi(t_1u_2u_1)\,\vf(u_1)\,du_1
\,dt_1\,du_2\\
&=\lim_{\ell\to+\infty}\int_{\varpi^{-\ell}\CO_\k}
\overline{\Psi}(t_1) \, \mu_0^{s_1}(t_1)\,
\Big(\int_{U_n}\CF_{{\bf k}}\big(\overline{\Psi}\mu_0^{s_2}\big)(u_2)\,\CF_\k(\vf)(t_1u_2)
\,du_2\Big)\,dt_1.
\end{align*}
 Let $D_x$, $x\in\k^\times$, be the dilation operator defined by 
$D_xf(t)=f(x^{-1}t)$. Since $\CF_{{\bf k}}\big(\overline{\Psi}\mu_0^{s_2}\big)$
is real valued and supported in $U_n$ (see Lemma \ref{corsigmaphi}) we have for $t_1\in\k^\times$:
\begin{align*}
\int_{U_n}\CF_{{\bf k}}\big(\overline{\Psi}\mu_0^{s_2}\big)(u_2)\,\CF_\k\big(\vf\big)(t_1u_2)
\,du_2&=|t_1|_\k^{-1}
\int_{\k}\CF_{{\bf k}}\big(\overline{\Psi}\mu_0^{s_2}\big)(t_2)\,\CF_\k\big(D_{t_1}\vf\big)(t_2)
\,dt_2\\
=|t_1|_\k^{-1}\big\langle \CF_{{\bf k}}\big(\overline{\Psi}\mu_0^{s_2}\big),\CF_\k\big(D_{t_1}\vf\big)\big\rangle
&=|t_1|_\k^{-1}\big\langle \overline{\Psi}\mu_0^{s_2},D_{t_1}\vf\big\rangle
=\big\langle D_{t_1^{-1}}\big(\overline{\Psi}\mu_0^{s_2}\big),\vf\big\rangle\\
=\int_{U_n}\Psi(t_1u_2)\,\mu_0^{s_2}(t_1u_2)\,\vf(u_2)\,du_2
&=
\mu_0^{s_2}(t_1)\int_{U_n}\Psi(t_1u_2)\,\vf(u_2)\,du_2=\mu_0^{s_2}(t_1)\,\CF_\k\big(\vf\big)(t_1),
\end{align*}
where we used Plancherel for $\k$ in the third equality and the invariance of $\mu_0$ under dilations 
in $\CO_\k^\times$ in the sixth.
Hence, we get
\begin{align*}
\big\langle\CF_{{\bf k}}\big(\overline{\Psi}\mu_0^{s_1}\big)\ast_{U_n}
\CF_{{\bf k}}\big(\overline{\Psi}\mu_0^{s_2}\big)\big|\vf\big\rangle&=
\lim_{\ell\to+\infty}\int_{\varpi^{-\ell}\CO_\k}
\overline{\Psi}(t_1) \, \mu_0^{s_1}(t_1)\,\mu_0^{s_2}(t_1)\,\CF_\k\big(\vf\big)(t_1)
\,dt_1\\
&=\int_{\k}
\overline{\Psi}(t_1) \, \mu_0^{s_1+s_2}(t_1)\,\CF_\k\big(\vf\big)(t_1)
\,dt_1=\big\langle\CF_\k\big(\overline{\Psi} \mu_0^{s_1+s_2}\big)\big|\vf\big\rangle.
\end{align*}

Assume now that $s_1\in\R$, $-1/2\leq s_2<0$ and chose $n\in\N$ large enough such that 
$(n+1)s_2<-1/2$ and $s_1+(n+1)s_2<-1/2$. From what precedes we have
\begin{align*}
\CF_{{\bf k}}\big(\overline{\Psi}\mu_0^{s_1}\big)\ast_{U_n}
\CF_{{\bf k}}\big(\overline{\Psi}\mu_0^{s_2}\big)&=
\CF_{{\bf k}}\big(\overline{\Psi}\mu_0^{s_1}\big)\ast_{U_n}
\CF_{{\bf k}}\big(\overline{\Psi}\mu_0^{(n+1)s_2}\big)\ast_{U_n}
\CF_{{\bf k}}\big(\overline{\Psi}\mu_0^{-ns_2}\big)\\&=
\CF_{{\bf k}}\big(\overline{\Psi}\mu_0^{s_1+(n+1)s_2}\big)\ast_{U_n}
\CF_{{\bf k}}\big(\overline{\Psi}\mu_0^{-ns_2}\big)=
\CF_{{\bf k}}\big(\overline{\Psi}\mu_0^{s_1+s_2}\big).
\end{align*}

Assume last  $s_1\in\R$, $s_2>0$ and chose $n\in\N$ large enough such that 
$-(n-1)s_2<-1/2$ and $-(n-1)s_2+s_1<-1/2$. From what precedes we have
\begin{align*}
\CF_{{\bf k}}\big(\overline{\Psi}\mu_0^{s_1}\big)\ast_{U_n}
\CF_{{\bf k}}\big(\overline{\Psi}\mu_0^{s_2}\big)&=
\CF_{{\bf k}}\big(\overline{\Psi}\mu_0^{s_1}\big)\ast_{U_n}
\CF_{{\bf k}}\big(\overline{\Psi}\mu_0^{-(n-1)s_2}\big)\ast_{U_n}
\CF_{{\bf k}}\big(\overline{\Psi}\mu_0^{ns_2}\big)\\&=
\CF_{{\bf k}}\big(\overline{\Psi}\mu_0^{s_1-(n-1)s_2}\big)\ast_{U_n}
\CF_{{\bf k}}\big(\overline{\Psi}\mu_0^{ns_2}\big)=
\CF_{{\bf k}}\big(\overline{\Psi}\mu_0^{s_1+s_2}\big),
\end{align*}
and the proof is complete.
\end{proof}

Let $\rho:X_n\to \CU(L^2(X_n))$, be the right regular representation. 
We are especially interested in the following family of right convolution operators on the group $X_n$:  
$$
J^{s}:=\int_{U_n} \CF_{{\bf k}}\big(\overline{\Psi}\mu_0^{s}\big)(u)\,\rho_{(u,[0])}\,du,
\quad  s\in\R.
$$ 

The next result justifies our choice of notation for the operators $J^s$:

\begin{prop}
\label{grp-conv}
Let $s,s_1,s_2\in \R$. The operator $J^s$ acts continuously on $\CD(X_n)$ and, moreover,
we have $J^0=\Id$ and $J^{s_1}\, J^{s_2}=J^{s_1+s_2}$.
\end{prop}
\begin{proof}
For the first part, we use  \cite[Proposition 7]{Bruhat} which
shows that the operator of convolution  by a compactly 
supported distribution on $X_n$ is continuous  on $\CD(X_n)$. The second part follows immediately from Lemma \ref{strange}.
\end{proof}

Since $\CF_{{\bf k}}\big(\overline{\Psi}\mu_0^{s}\big)$ is real valued and invariant under the
group inversion in $U_n$, one easily deduce that $J^s$ is formally self-adjoint. In fact, it is
not difficult to see that on the domain $\CD(X_n)$, $J^s$ is self-adjoint. Moreover, Proposition
\ref{grp-conv} shows that $J^s=J^{\frac s2}J^{\frac s2}$  on $\CD(X_n)$ and therefore 
$J^s$ is non-negative.

By transposition, the operators $J^s$, $s\in\R$, act by homeomorphisms on $\CD'(X_n)$. 
Now, for $x\in\k$ fixed, we let $\widehat \Psi_x\in L^\infty(X_n)$ be defined by $\widehat \Psi_x(u,[t]):=
\Psi_x(u)=\Psi(xu)$. As a distribution on $X_n$, $\widehat\Psi_x$ is an eigenfunction of $J^s$
for all $s\in\R$:
\begin{lem}
\label{eigenPsi}
For every $x\in\k$ and $s\in\R$, we have
$$
J^s\widehat\Psi_x=\mu_0^s(x)\,\widehat\Psi_x.
$$
\end{lem}
\begin{proof}
We may assume without lost of generality that $s<-1/2$. Indeed,
once this case has been proven, we will then have for $s>1/2$ and $\vf\in\CD(X_n)$:
$$
\langle J^s\widehat\Psi_x|\vf\rangle=\langle\widehat\Psi_x|J^s\vf\rangle=\langle\widehat\Psi_x|J^{-s}J^{2s}\vf\rangle
=\langle J^{-s}\widehat\Psi_x|J^{2s}\vf\rangle=
\mu_0^{-s}(x)\langle\widehat\Psi_x|J^{2s}\vf\rangle=\mu_0^{-s}(x)\langle J^{2s}\widehat\Psi_x|\vf
\rangle.
$$ 
Hence $J^s\widehat\Psi_x=\mu_0^{-s}(x)J^{2s}\widehat\Psi_x$ and 
applying $J^{-s}$ on both sides gives
$\widehat\Psi_x=\mu_0^{-s}(x)J^{s}\widehat\Psi_x$.
Thus, if the result holds for $s<-1/2$ then it also holds for $s>1/2$. 
Moreover, if the result holds for $|s|>1/2$ then it also holds for $0<|s|\leq1/2$.
Indeed, for $0<s\leq1/2$ chose $n\in\N$ so large that $(n-1)s>1/2$. Then we have
$J^s\widehat\Psi_x=J^{ns}J^{(-n+1)s}\widehat\Psi_x=
\mu_0(x)^{(-n+1)s}
J^{ns}\widehat\Psi_x=\mu_0(x)^{(-n+1)s}
\mu_0(x)^{ns}\widehat\Psi_x$. Passing from $0<s\leq1/2$ to $-1/2\leq s<0$ as
before, we are done.

So we just need to consider  $s<-1/2$.
In this case, $\mu_0^s\in L^2(\k)$ and thus $\CF_{{\bf k}}\big(\overline{\Psi}\mu_0^{s}\big)
\in L^2(\k)\cap L^1(\k)$. (Since $\CF_{{\bf k}}\big(\overline{\Psi}\mu_0^{s}\big)$
  is supported on $U_n$.) 
We then have
\begin{align*}
J^s\widehat\Psi_x(u,[t])&=\int_{U_n}\CF_{{\bf k}}\big(\overline{\Psi}\mu_0^{s}\big)(u_0)
\widehat\Psi_x\big((u,[t])(u_0,[0])\big)\,du_0\\
&=\int_{U_n}\CF_{{\bf k}}\big(\overline{\Psi}\mu_0^{s}\big)(u_0)
\widehat\Psi_x\big(uu_0,[u_0^{-1}t]\big)\,du_0
=\int_{U_n}\CF_{{\bf k}}\big(\overline{\Psi}\mu_0^{s}\big)(u_0)
\Psi(xuu_0)\,du_0.
\end{align*}
Viewing $\CF_{{\bf k}}\big(\overline{\Psi}\mu_0^{s}\big)$ as an $L^2$-function on $\k$
supported on $U_n$, we get
\begin{align*}
J^s\widehat\Psi_x(u,[t])&
=\int_{\k}\CF_{{\bf k}}\big(\overline{\Psi}\mu_0^{s}\big)(t_0)
\Psi(xut_0)\,dt_0=\CF_\k^{-1}\CF_{{\bf k}}\big(\overline{\Psi}\mu_0^{s}\big)(-xu)=
\Psi(xu)\mu_0^s(-xu)=\mu_0^s(x)\widehat\Psi_x(u,[t]),
\end{align*}
where the last equality follows by invariance of $\mu_0$ by dilations in $\CO_\k^\times$.
\end{proof}

We now introduce the following
subspace of distributions on $X_n$:
$$
\CB(X_n) := \Big \{  F \in \CD'(X_n) : \forall j  \in\N, \, J^jF \in L^{\infty}(X_n) \Big \}.
$$
We endow $\CB(X_n)$ with the topology associated with the family of seminorms:
\begin{equation}
F\mapsto \|J^j F\|_\infty, \;\;\; \forall j \in\N. 
\end{equation}
The main properties of the space $\CB(X_n)$ are summarized in the next proposition. There,
we denote by $C_{ru}(X_n)$ the $C^*$-algebra of bounded and right-uniformly continuous
complex valued functions on $X_n$ (our convention for the right-uniform structure is the one 
which yields strong continuity for the right regular action) and by $C_{ru}^\infty(X_n)$ the subspace
of $C_{ru}(X_n)$ on which the right action is regular (or smooth) in the sense of Bruhat.
The proof of the following  result is very close  to \cite[Lemma 3.2]{GJ1} but for the sake of completeness 
we give here the detailed arguments. 
\begin{prop} 
\label{frechetfonct}
The space $\CB(X_n)$ is Fr\'echet and 
we have $C_{ru}^\infty(X_n)\subset \CB(X_n) \subset C_{ru}(X_n)$ with dense 
inclusions. 
\end{prop}
\begin{proof}
That $\CB(X_n)$ is Fr\'echet follows from standard arguments.
By   \cite[(32.45) (b), p. 283]{HR2}, it follows that the space of left-uniformly continuous and bounded
functions on $X_n$ is exactly $L^1(X_n)\ast_{X_n} L^\infty(X_n)$, where $\ast_{X_n}$
denotes the convolution product on $X_n$. (Be aware that for  uniform
structures, we use the convention opposite to those of \cite{HR2}.) Hence, with $S$ 
antipode of $L^\infty(X_n)$ (i.e$.$ $SF([g])=F([g]^{-1})$), we need to show that
$S\CB(X_n)\subset L^1(X_n)\ast_{X_n} L^\infty(X_n)$. Define 
$\Phi:=\CF_{{\bf k}}\big(\overline{\Psi}\mu_0^{-2})\otimes  \delta_{[0]}\in L^1(X_n)$. Then we have
for any $F\in\CB(X_n)$, and with $\lambda$ and $\rho$ the left and right regular representations (possibly in their integrated version), we get since
$J^{-2}=\rho(\Phi)$:
$$
SF=SJ^{-2}J^2F=S\rho(\Phi)J^2F=\lambda(\Phi)  SJ^2F=\Phi\ast_{X_n} SJ^2F.
$$
Since $J^2F\in\CB(X_n)\subset L^\infty(X_n)$, we deduce that $SJ^2F\in L^\infty(X_n)$
which entails that $\CB(X_n)\subset C_{ru}(X_n)$. 
The inclusion $C_{ru}^\infty(X_n)\subset \CB(X_n)$ 
follows from the Dixmier-Malliavin theorem, proven  for arbitrary locally compact groups
by Meyer  in \cite[Theorem 4.16]{Meyer}. Indeed, the latter states that $C_{ru}^\infty(X_n)$ coincides with
its Garding space, that is the space of finite sums of elements of the form 
$$
\rho(f)F:=\int_{X_n}f([g])\,\rho_{[g]}(F)\,d[g]\;,\qquad f\in\CD(X_n)\,,\;F\in C_{ru}(X_n).
$$
For $j\in\N$ and $f\in\CD(X_n)$, we have
$$
J^j \rho(f)=\int_{U_n\times X_n} \CF_{{\bf k}}\big(\overline{\Psi}\mu_0^{j}\big)(u)\,f([g])\,\rho_{(u,[0])[g]}\,dud[g]
=\int_{U_n\times X_n} \CF_{{\bf k}}\big(\overline{\Psi}\mu_0^{j}\big)(u)\,f((u^{-1},[0])[g])\,\rho_{[g]}\,dud[g],
$$
we easily  deduce  that $J^j \rho(f)=\rho(SJ^jSf)$. Hence, we get
$$
\|J^j \rho(f)F\|_\infty=\|\rho(SJ^jSf)F\|_\infty\leq \|SJ^jSf\|_1\|F\|_\infty<\infty,
$$
and therefore $C_{ru}^\infty(X_n)\subset \CB(X_n)$.
Finally, density of $C_{ru}^\infty(X_n)$ in $C_{ru}(X_n)$ is a consequence of the strong continuity of the right regular action on the $C^*$-algebra of right-uniformly
continuous and bounded functions on $X_n$.
\end{proof}

The arguments of the next result  are conceptually similar to
those of \cite[Corollary 3.6]{GJ1} but technically different.
\begin{prop} 
The space $\CB(X_n)$ is a Fr\'echet algebra under the point-wise product. More precisely, for all $j \in\N$ and all 
$F_1, \, F_2 \in \CB(X_n)$, we have:
\begin{equation}
\label{SNEST}
\big\|J^j(F_1 \, F_2)\big\|_\infty \leq q^{-2n}
\|\mu_0^{-2}\|_1^2 \, \big\|J^{j+2}F_1\big\|_\infty \,  \big\|J^{j+2}F_2\big\|_\infty. 
\end{equation}
\end{prop}
\begin{proof}

Let $j \in\N$, $F_1,  F_2 \in \CB(X_n)$ and $(u_0,[t_0]) \in X_n$. Since the function $\CF_{\k}
(\overline{\Psi_1} \, \mu_0^{-j-2})\in L^1(U_n)$ is invariant under group inversion (by Lemma \ref{corsigmaphi}),
 we get:
\begin{align*}
J^{-j-2}F_k(u_0,[t_0])  
&= \int_{U_n }\CF_{{\bf k}}\big(\overline{\Psi}\mu_0^{-j-2}\big)(u)\,F_k(u_0u,[u^{-1}t_0])\,du\\
&=\int_{U_n }\CF_{{\bf k}}\big(\overline{\Psi}\mu_0^{-j-2}\big)(u^{-1}u_0)\,
F_k(u,[u^{-1}u_0t_0])du\\
&=\int_{U_n \times\k}\Psi(u^{-1}u_0t) \, \overline\Psi(t)\,\mu_0^{-j-2}(t)\,F_k(u,[u^{-1}u_0t_0])\,du\,dt.
\end{align*}
Writing $F_1=J^{-j-2} J^{j+2}F_1$,  $F_2=J^{-j-2} J^{j+2}F_2$, it then follows that 
\begin{align*}
 F_1 \, F_2(u_0,[t_0])  
 =  \int_{(U_n \times \k)^2}  &  \Psi\big( u_0 (u_1^{-1}t_1 + u_2^{-1}t_2 ) \big) \,  
 \overline{\Psi}(t_1+t_2) \,  \mu_0^{-j-2}(t_1) \, \mu_0^{-j-2}(t_2)  \\ 
 & \times J^{j+2}F_1(u_1,[u_1^{-1}u_0t_0])  \,  J^{j+2}F_2(u_2,[u_2^{-1}u_0t_0])  \, 
 du_1dt_1du_2dt_2.
\end{align*}

To compute  $J^j(F_1F_2)$ we let $J^j$ acts on this absolutely convergent 
(in $L^\infty(X_n)$) integral 
representation of the product $F_1F_2$.  Hence, we are left to compute
$$
J^j\Big[(u_0,[t_0])\mapsto \Psi\big( u_0 (u_1^{-1}t_1 + u_2^{-1}t_2 ) \big)\,
J^{j+2}F_1(u_1,[u_1^{-1}u_0t_0])\,J^{j+2}F_2(u_2,[u_2^{-1}u_0t_0])\Big].
$$
Now observe that an element  $F\in L^\infty(X_n)$  of the form $F(u_0,[t_0])=h([uu_0t_0])$, 
with $h\in\ell^\infty(\Gamma_n)$
and $u\in\CO_\k^\times$,  is invariant under right translations in  $\{(u,[0]),\;u\in U_n\}\subset X_n$.
Hence,  we get
\begin{align*}
J^j\big( F_1 \, F_2\big)(u_0,[t_0])  
 =  \int_{(U_n \times \k)^2}  & J^j\big( \widehat\Psi_{u_1^{-1}t_1 + u_2^{-1}t_2}\big) (u_0,[t_0]) \,  
 \overline{\Psi}(t_1+ t_2) \,  \mu_0^{-j-2}(t_1) \, \mu_0^{-j-2}(t_2)  \\ 
 & \times J^{j+2}F_1(u_1,[u_1^{-1}u_0t_0])  \,  J^{j+2}F_2(u_2,[u_2^{-1}u_0t_0])  \, 
 du_1dt_1du_2dt_2,
\end{align*}
which by Lemma \ref{eigenPsi} entails
\begin{align*}
J^j\big( F_1 \, F_2\big)(u_0,[t_0])  
&=  \int_{(U_n \times \k)^2}  \mu_0^j\big(u_0(u_1^{-1}t_1 + u_2^{-1}t_2)\big)\,
\Psi\big(u_1^{-1}t_1 + u_2^{-1}t_2\big) \,  
 \overline{\Psi}(t_1+t_2)\,\mu_0^{-j-2}(t_1)  \\ 
 &\qquad \times   \mu_0^{-j-2}(t_2) \,J^{j+2}F_1(u_1,[u_1^{-1}u_0t_0]_0)  \,  
 J^{j+2}F_2(u_2,[u_2^{-1}u_0t_0])  \, 
 du_1dt_1du_2dt_2.
\end{align*}
This formula then yields the following estimate 
\begin{align*}
&\big|J^j\big( F_1 \, F_2\big)\big|(u_0,[t_0])\\
&\quad  \leq\|J^{j+2}F_1\|_\infty \|J^{j+2}F_2\|_\infty
  \int_{(U_n \times \k)^2}&  \mu_0^j\big(u_0(u_1^{-1}t_1 + u_2^{-1}t_2)\big)\,
  \mu_0^{-j-2}(t_1) \, \mu_0^{-j-2}(t_2)  
 du_1dt_1du_2dt_2.
\end{align*}
Using the Peetre inequality  \eqref{Peetre} and the invariance of  $\mu_0$ by dilations in 
$\CO_\k^\times$, we get
$$
\mu_0^j\big(u_0(u_1^{-1}t_1 + u_2^{-1}t_2)\big) \leq \mu_0^j(u_0u_1^{-1}t_1)\,
\mu_0^j(u_0u_2^{-1}t_2)=\mu_0^j(t_1)\,
\mu_0^j(t_2).
$$
With our choice of normalization for Haar measures, we have ${\rm Vol}(U_n)=q^{-n}$
and thus
\begin{align*}
\|J^j (F_1 \, F_2 )\|_\infty & \, \leq  \, q^{-2n} \, ||\mu_0^{-2}||_1^2 \,\|J^{j+2}F_1\|_\infty
 \|J^{j+2}F_2\|_\infty,
\end{align*}
which is the inequality we had to prove.
\end{proof}

Even though we shall not need this space here (we will use it in \cite{GJ3} only),  we conclude this paragraph
by constructing a more suitable (for us) version of the Schwartz space  on the group $X_n$. The idea is to control
 regularity  using the operator $J$ and    decay using the 
multiplication operator:
$$
I\vf(u,[t]):=\mu_0(t)\,\vf(u,[t]).
$$
(Recall that $\mu_0$ is invariant under translations in $\varpi^{-n}\CO_\k$ so that it defines a 
function on $\Gamma_n$.)
With domain $\CD(X_n)\subset L^2(X_n)$, the operator $I$ is unbounded, essentially self-adjoint and non-negative.
 Since moreover $\mu_0(t)=\mu_0(ut)$ for all $u\in U_n$, one sees that $I$
commutes with $J$. We then let
$$
\CS(X_n):= \Big \{  f \in \CB(X_n) : \forall j  \in\N, \, I^jf \in \CB(X_n) \Big \},
$$
and we endow $\CS(X_n)$ with the topology associated with the family of seminorms:
\begin{equation}
\label{ss}
f\mapsto\|J^kI^j f\|_\infty, \;\;\; \forall j,k \in\N. 
\end{equation}
From standard methods, one can prove
that, endowed with this topology, $\CS(X_n)$ is Fr\'echet and nuclear. Since
moreover $I^j(fF)=(I^jf)F$,  for $f\in\CS(X_n)$ and $F\in\CB(X_n)$ we also get from 
\eqref{SNEST} that
$$
\|J^kI^j(fF)\|_\infty\leq  q^{-2n}\|\mu_0^{-2}\|_1^2 \, \|J^{k+2}I^jf\|_\infty\,
\|J^{k+2}F\|_\infty.
$$
Hence $\CS(X_n)$ is an ideal of $\CB(X_n)$ (for the point-wise product).

\subsection{The main estimate}
The  goal of this part is to provide an analogue of the Calder\'on-Vaillancourt Theorem
for the $p$-adic Fuchs calculus. 
Namely, we will prove 
 that the quantization map 
$\bO_\theta:\CD'(X_n)\to\CL\big(\CD(U_n),\CD'(U_n)\big)$,  restricts to $\CB(X_n)$   as a bounded operator on $L^2(U_n)$.  
The method we use relies on  coherent states and Wigner functions, a method which has been discovered  by Unterberger in the eighties \cite{Unold} in the context of the Weyl calculus. 
To this aim, let us consider the following specific Wigner functions:
\begin{align}
\label{WS}
W_{g}^{\theta}    := W_{\mathds{1}_{U_n}, \, \pi_{\theta}(g)\mathds{1}_{U_n} }\in L^2(X_n), \quad g\in G_n,
\end{align}
where $\mathds{1}_{U_n}$ is the characteristic function of $U_n$.
 From \eqref{wig1} and \eqref{Utheta}
we explicitly get with $g_1=(u_1,t_1)\in G_n$ and $[g_2]=(u_2,[t_2])\in X_n$:
$$
W_{g_1}^{\theta}([g_2])=\int_{U_n} \Psi_\theta\big(\phi(u_2u_0^{-1})t_2\big)\,
\Psi_\theta\big(u_0u_2^{-2}u_1t_1\big)\,du_0
=\int_{U_n} 
\Psi_\theta\big(u_0u_2^{-1}u_1t_1-\phi(u_0)t_2\big)\,du_0.
$$
Since $\mathds{1}_{U_n}\in\CD(U_n)$, we may use Proposition \ref{wignerbruhat}  to get that  
  $W_{g}^{\theta} \in \CD(X_n)$.  
Hence, we can act on $W_{g}^{\theta} $ by the operator $J^s$:
\begin{lem}
\label{intmu}
Let $\theta\in\CO_\k^\times$,
$s<-1/2$, $u_1,u_2\in U_n$ and $t_1,t_2\in \k$.
Then, for $g_1=(u_1,t_1) \in G_n$ and  $[g_2]=(u_2,[t_2]) \in X_n$, we have 
$$
J^{s}W_{g_1}^{\theta}([g_2]) =    \int_{U_n  } 
 \mu_0^{s}\big(u_0u_2^{-1}u_1t_1 -\phi(u_0)t_2 \big)   \,   
 \Psi_\theta\big(u_0u_2^{-1}u_1t_1-\phi(u_0)t_2\big)\,du_0. 
 $$
\end{lem}
\begin{proof}
Note that since $s<-1/2$,      $\CF_\k\big(\overline\Psi \mu_0^s\big)\in L^2(\k)\cap L^1(\k)$.
By definition, we have:
\begin{align}
\label{label}
J^{s}W_{g_1}^{\theta}([g_2])&=\int_{U_n}\CF_\k\big(\overline\Psi \mu_0^s\big)(u_3)\,
W_{g_1}^{\theta}((u_2,[t_2]),(u_3,[0]))\,du_3\nonumber\\
&=\int_{U_n}\CF_\k\big(\overline\Psi \mu_0^s\big)(u_3)\,
W_{g_1}^{\theta}(u_2u_3,[u_3^{-1}t_2])\,du_3\nonumber\\
&=\int_{U_n\times U_n}\CF_\k\big(\overline\Psi \mu_0^s\big)(u_3)\,
\Psi_\theta\big(u_3^{-1}(u_0u_2^{-1}u_1t_1-\phi(u_0)t_2)\big)\,du_0\,du_3\nonumber\\
&=\int_{U_n\times U_n}\CF_\k\big(\overline\Psi \mu_0^s\big)(u_3)\,
\Psi_\theta\big(u_3(u_0u_2^{-1}u_1t_1-\phi(u_0)t_2)\big)\,du_0\,du_3,
\end{align}
where  the last equality follows from the invariance of $\CF_\k\big(\overline\Psi \mu_0^s\big)$
under group inversion (see Lemma \ref{corsigmaphi}).
Now, if we view $\CF_\k\big(\overline\Psi \mu_0^s\big)$ as an $L^2$-function on $\k$
supported on $U_n$, we get:
\begin{align*}
&\int_{U_n}\!\CF_\k\big(\overline\Psi \mu_0^s\big)(u_3)\,
\Psi_\theta\big(u_3(u_0u_2^{-1}u_1t_1-\phi(u_0)t_2)\big)\,du_3
=\int_\k\!\CF_\k\big(\overline\Psi \mu_0^s\big)(t_3)\,
\Psi_\theta\big(t_3(u_0u_2^{-1}u_1t_1-\phi(u_0)t_2)\big)\,dt_3\\
&=\CF_\k^{-1}\CF_\k\big(\overline\Psi \mu_0^s\big)
\big(-\theta(u_0u_2^{-1}u_1t_1-\phi(u_0)t_2)\big)=
\Psi_\theta\big(u_0u_2^{-1}u_1t_1-\phi(u_0)t_2\big)\,
 \mu_0^s\big(u_0u_2^{-1}u_1t_1-\phi(u_0)t_2\big),
\end{align*}
 by invariance of $\mu_0$ by dilations in $\CO_\k^\times$. This concludes the proof.
\end{proof}

\begin{lem}
\label{supp}
Let $\theta\in\CO_\k^\times$,
 $s<-1/2$, $g_1=(u_1,t_1)\in G_n$ and $u_2\in U_n$. Then the map 
$\Gamma_n\to\C$, $[t_2] \mapsto J^{s}W_{g_1}^{\theta}  (u_2,[t_2]) $ 
is supported in the finite set $ \varpi^{\min(-n,{\rm val}(t_1) )}\CO_{\bf k}/\varpi^{-n}\CO_{\k}$.
\end{lem}
\begin{proof}
Using the substitution formula \eqref{change}, we deduce from \eqref{label}:
\begin{align*}
J^{s}W_{g_1}^{\theta}([g_2])&=
\int_{\varpi^n\CO_\k\times U_n}\CF_\k\big(\overline\Psi \mu_0^s\big)(u_3)\,
\Psi_\theta\big(u_3(\phi^{-1}(u_0)u_2^{-1}u_1t_1-u_0t_2)\big)\,du_0\,du_3\\
&=\int_{U_n}\CF_\k\big(\overline\Psi \mu_0^s\big)(u_3)\,
\CF_\k\big(\mathds{1}_{\varpi^n\CO_\k}
\Psi_{\theta u_3u_2^{-1}u_1t_1}\circ\phi^{-1}\big)(u_3t_2)\,du_3.
\end{align*}
Observe first that $\Psi_{\theta u_3u_2^{-1}u_1t_1}\circ\phi^{-1}$ is constant on the cosets
of $\varpi^{-{\rm val}(t_1)}\CO_\k$. Indeed, for $x\in\CO_\k$ and $t\in\k$, we have
since $\phi$ is an isometry:
\begin{align*}
\big|u_3u_2^{-1}u_1t_1\phi^{-1}\big(t+\varpi^{-{\rm val}(t_1)}x\big)
- u_3u_2^{-1}u_1t_1\phi^{-1}(t)\big|_\k&=|t_1|_\k\,\big|\phi^{-1}\big(t+\varpi^{-{\rm val}(t_1)}x\big)
-\phi^{-1}(t)\big|_\k\\&=|t_1|_\k\, |\varpi^{-{\rm val}(t_1)}|_\k\,|x|_\k=|x|_\k\leq 1,
\end{align*}
which prove this first claim since $\Psi_\theta$ is constant in $\CO_\k$.
Hence, $\mathds{1}_{\varpi^n\CO_\k}\Psi_{\theta u_3u_2^{-1}u_1t_1}\circ\phi^{-1}$ is invariant under
translations in $ \varpi^{\max(n,-{\rm val}(t_1) )}\CO_{\bf k}$ which implies that
$\CF_\k\big(\mathds{1}_{\varpi^n\CO_\k}
\Psi_{\theta u_3u_2^{-1}u_1t_1}\circ\phi^{-1}\big)$ is supported on  
$\varpi^{-\max(n,-{\rm val}(t_1) )}\CO_{\bf k}$ and, since $u_3$ is a unit, so is the map
$t_2\mapsto\CF_\k\big(\mathds{1}_{\varpi^n\CO_\k}
\Psi_{\theta u_3u_2^{-1}u_1t_1}\circ\phi^{-1}\big)(u_3t_2)$.
\end{proof}

The previous Lemmas allow to prove a crucial property of the Wigner functions 
$W_{g}^{\theta}$, $g\in G_n$:

\begin{prop}  
\label{maj} 
Let $\theta\in\CO_\k^\times$ and $s<-2$. Then,  
 $$  \int_{G_n \times X_n} \big | J^{s}W_{g_1}^{\theta}([g_2])   \big  | \, dg_1d[g_2]
 \leq q^{-2n}\big(1+q^{-n}\,\|
   \mu_0^{s+1}\|_1\big) < \infty.$$
\end{prop}

\begin{proof}
Viewing $[g_2]\mapsto J^sW_{g_1}^{\theta}([g_2])$ as a function on 
$G_n$ constant on the cosets of $H_n=\{1\}\times \varpi^{-n}\CO_\k\subset G_n$, 
we first write
$$
 \int_{G_n \times X_n} \big | J^{s}W_{g_1}^{\theta}([g_2])   \big  | \, dg_1d[g_2] =
 q^{-n} \int_{G_n \times G_n} \big | J^{s}W_{g_1}^{\theta}([g_2])   \big  | \, dg_1dg_2.
$$
Next, we split the integration domain as the disjoint union of $V_1:=\{(g_1,g_2)\in G_n\times G_n\;:
\; |t_2|_\k\leq|t_1|_\k\}$ with $V_2:=\{(g_1,g_2)\in G_n\times G_n\;:
\; |t_2|_\k>|t_1|_\k\}$ and we denote by $I_1$ and $I_2$ the corresponding integrals.

First, by Lemma \ref{intmu}, we get
\begin{align*}
I_1&= q^{-n} \, \int_{V_1}  \Big|\int_{U_n  } 
 \mu_0^{s}\big(u_0u_2^{-1}u_1t_1 -\phi(u_0)t_2 \big)   \,   
 \Psi_\theta\big(u_0u_2^{-1}u_1t_1-\phi(u_0)t_2\big)\,du_0\Big|\,du_1du_2dt_1dt_2.
\end{align*}
Observe next that when $(g_1,g_2)\in V_1$ we have $|u_0u_2^{-1}u_1t_1|_\k=|t_1|_\k \geq |t_2|_\k>|\phi(u_0)t_2|_\k$ since $|\phi(u_0)|_\k\leq q^{-n}<1$. The case
of equality in the ultrametric triangle inequality and the invariance of $\mu_0$ under dilations in $\CO_\k$ therefore give: 
$$
\mu_0^{s}\big(u_0u_2^{-1}u_1t_1 -\phi(u_0)t_2 \big)=\mu_0^{s}\big(u_0u_2^{-1}u_1t_1\big)
=\mu_0^{s}(t_1).
$$
Hence, since $|t|_\k\leq q^n\mu_0(t)$, 
\begin{align*}
I_1&= q^{-n} \, \int_{V_1}  \Big|\int_{U_n  } 
 \mu_0^{s}(t_1)   \,   
 \Psi_\theta\big(u_0u_2^{-1}u_1t_1-\phi(u_0)t_2\big)\,du_0\Big|\,du_1du_2dt_1dt_2\\
 &\leq q^{-4n}  \int_\k \mu_0^{s}(t_1)\Big(\int_{\varpi^{{\rm val}(t_1)}\CO_\k}
   \,   dt_2\Big)dt_1
   =q^{-4n}  \int_\k \mu_0^{s}(t_1)\,|t_1|_\k\,dt_1\leq q^{-3n}\,\|
   \mu_0^{s+1}\|_1,
\end{align*}
which is finite since $s<-2$.

Next, if $(g_1,g_2)\in V_2$, we have $ |t_1|_\k< |t_2|_\k$. But by Lemma \ref{supp}
we also have $|t_2|_\k\leq \max(q^n,|t_1|_\k)$ when $J^s W_{g_1}^{\theta}([g_2])$ is possibly
nonzero. Hence $|t_2|_\k\leq q^n$ and thus $|t_1|_\k\leq q^n$ too and 
 we deduce that the domain of integration  of the variables $t_1$ and $t_2$ reduces
to $\varpi^{-n}\CO_\k$ when $J^s W_{g_1}^{\theta}([g_2])$ is possibly
nonzero. Thus
\begin{align*}
I_2&= q^{-n} \, \int_{V_2}  \Big|\int_{U_n  } 
 \mu_0^{s}\big(u_0u_2^{-1}u_1t_1 -\phi(u_0)t_2 \big)   \,   
 \Psi_\theta\big(u_0u_2^{-1}u_1t_1-\phi(u_0)t_2\big)\,du_0\Big|\,du_1du_2dt_1dt_2\\
 &\leq q^{-n} \, \int_{U_n\times U_n\times U_n\times
  \varpi^{-n}\CO_\k\times \varpi^{-n}\CO_\k}  
 \mu_0^{s}\big(u_0u_2^{-1}u_1t_1 -\phi(u_0)t_2 \big)   \,   
 du_0du_1du_2dt_1dt_2\\
& \leq q^{-n}{\rm Vol}(U_n)^3{\rm Vol}(\varpi^{-n}\CO_\k)^2=q^{-2n},
\end{align*}
where we used that $\mu_0\geq 1$ and $s<0$.
\end{proof}

 Recall that $\mathds{1}_{U_n}$  denotes the characteristic function of ${U_n}$.
 
\begin{cor}
\label{wignersup} 
Let $\theta\in\CO_\k^\times$ and $s<-2$. Then
for every $F \in \CB(X_n)$, we have:
$$
\sup_{g_1 \in G_n} \int_{G_n}  \big |  \big\langle \pi_{\theta}(g_1)\mathds{1}_{U_n}, \,\bO_{\theta}(F) 
\pi_{\theta}(g_2) \mathds{1}_{U_n} \big\rangle   \big |  \, dg_2 \leq q^{-n}\big(1+q^{-n}\,\|
   \mu_0^{s+1}\|_1\big)\,\|J^{-s}F\|_\infty<\infty.
$$
\end{cor}

\begin{proof}
Since $\pi_{\theta}(g)\mathds{1}_{U_n}\in\CD(U_n)$ for all $g\in G_n$,  the function
$(g_1,g_2)\mapsto\big\langle \pi_{\theta}(g_1)\mathds{1}_{U_n}, \,\bO_{\theta}(F) 
\pi_{\theta}(g_2) \mathds{1}_{U_n} \big\rangle$ is well defined and by Proposition \ref{extend} we have
$$
\big\langle \pi_{\theta}(g_1)\mathds{1}_{U_n}, \,\bO_{\theta}(F) 
\pi_{\theta}(g_2) \mathds{1}_{U_n} \big\rangle= q^n
\int_{X_n} F([g])\,W_{\pi_{\theta}(g_1) \mathds{1}_{U_n}, \, \pi_{\theta}(g_2) \mathds{1}_{U_n}}^{\theta}([g])\,d[g].
$$
Using  $G_n$-covariance, we  deduce
\begin{align*}
W_{\pi_{\theta}(g_1) \mathds{1}_{U_n}, \, \pi_{\theta}(g_2) \mathds{1}_{U_n}}^{\theta}([g])&=
\big\langle \pi_{\theta}(g_1) \mathds{1}_{U_n},\Omega_\theta([g]) \pi_{\theta}(g_2) \mathds{1}_{U_n}\big\rangle
=
\big\langle \mathds{1}_{U_n},\pi_{\theta}(g_1^{-1})\Omega_\theta([g]) \pi_{\theta}(g_2) \mathds{1}_{U_n}\big\rangle\\
&=
\big\langle \mathds{1}_{U_n},\Omega_\theta([g_1^{-1}g]) \pi_{\theta}(g_1^{-1}g_2) \mathds{1}_{U_n}\big\rangle
=W^\theta_{g_1^{-1}g_2}([g_1^{-1}g]).
\end{align*}
Thus,
\begin{align*}
\big\langle \pi_{\theta}(g_1)\mathds{1}_{U_n}, \,\bO_{\theta}(F) 
\pi_{\theta}(g_2) \mathds{1}_{U_n} \big\rangle&=q^n
\int_{X_n} F([g])\,W^\theta_{g_1^{-1}g_2}([g_1^{-1}g])\,d[g]\\&
=q^n\int_{X_n} \big(\lambda_{[g_1]^{-1}}F\big)([g])\,W^\theta_{g_1^{-1}g_2}([g])\,d[g].
\end{align*}
Take then $s<-2$. Since $J^{-s}$ commutes with left translations on the group $X_n$, we therefore have
\begin{align}
\label{formula}
\big\langle \pi_{\theta}(g_1)\mathds{1}_{U_n}, \,\bO_{\theta}(F) 
\pi_{\theta}(g_2) \mathds{1}_{U_n} \big\rangle
=q^n\int_{X_n} \big(\lambda_{[g_1]^{-1}}J^{-s}F\big)([g])\,\big(J^s W^\theta_{g_1^{-1}g_2}\big)
([g])\,d[g].
\end{align}
This implies that
\begin{align*}
&\sup_{g_1 \in G_n} \int_{G_n}  \big |  \big\langle \pi_{\theta}(g_1)\mathds{1}_{U_n}, \,\bO_{\theta}(F) 
\pi_{\theta}(g_2) \mathds{1}_{U_n} \big\rangle   \big |  \, dg_2\\
&\qquad\qquad=q^n
\sup_{g_1 \in G_n} \int_{G_n}  \Big |  \int_{X_n} \big(\lambda_{[g_1]^{-1}}J^{-s}F\big)([g])\,\big(J^
s W^\theta_{g_1^{-1}g_2}\big) ([g])\,d[g]\Big |  \, dg_2\\
&\qquad\qquad=q^n\sup_{g_1 \in G_n}\int_{G_n}  \Big |  \int_{X_n} \big(\lambda_{[g_1]^{-1}}
J^{-s}F\big)([g])\,
\big(J^s W^
\theta_{g_2}\big)([g])\,d[g] \Big |  \, dg_2\\
&\qquad\qquad\leq q^n  \sup_{g_1 \in G_n}\|\lambda_{[g_1]^{-1}}J^{-s}F\|_\infty
\int_{X_n\times G_n} \big|\big(J^s W^\theta_{g_2}\big)([g])\big|\,d[g]    dg_2\\
&\qquad\qquad\quad=q^n\|J^{-s}F\|_\infty
\int_{X_n\times G_n} \big|\big(J^s W^\theta_{g_2}\big)([g])\big|\,d[g]    dg_2,
\end{align*}
and the claim follows by Proposition \ref{maj}.
\end{proof}

We are now  ready to state our version of the Calder\'on-Vaillancourt inequality  for the
 $p$-adic Fuchs calculus. This is an immediate application of the Schur test Lemma in the context
of square integrable irreducible unitary representations,  together with the fundamental inequality 
given in Corollary \ref{wignersup}. Indeed, for $\vf,\psi\in\CD(U_n)\subset L^2(U_n)$, using twice the resolution of the identity
given in Corollary \ref{ident} for the mother wavelet $\mathds{1}_{U_n}$ (the characteristic function of $U_n$), we have:
\begin{align*}
\big\langle\psi,\bO_\theta(F)\vf\big\rangle&=\|\mathds{1}_{U_n}\|_2^{-4}\int_{G_n\times G_n}\langle\psi,\pi_{\theta}(g_1)\mathds{1}_{U_n}\rangle 
\langle \pi_{\theta}(g_1)\mathds{1}_{U_n}, \,\bO_{\theta}(F)
\pi_{\theta}(g_2) \mathds{1}_{U_n} \big\rangle  \langle\pi_{\theta}(g_2)\mathds{1}_{U_n},\psi\rangle\,dg_1dg_2.
\end{align*} 
Using then the Cauchy-Schwarz inequality and  Corollary \ref{ident} backward, we get
\begin{align*}
\big|\big\langle\psi,\bO_\theta(F)\vf\big\rangle\big|&\leq\|\mathds{1}_{U_n}\|_2^{-2}\|\vf\|_2\|\psi\|_2 \Big(\sup_{g_1}\int_{G_n\times G_n} 
\big|\langle \pi_{\theta}(g_1)\mathds{1}_{U_n}, \,\bO_{\theta}(F)
\pi_{\theta}(g_2) \mathds{1}_{U_n} \big\rangle \big|dg_2\Big)^{1/2}\\
&\qquad\quad\qquad\qquad\qquad\times \Big(\sup_{g_2}\int_{G_n\times G_n} 
\big|\langle \pi_{\theta}(g_1)\mathds{1}_{U_n}, \,\bO_{\theta}(F)
\pi_{\theta}(g_2) \mathds{1}_{U_n} \big\rangle \big|dg_1\Big)^{1/2}.
\end{align*} 
Since moreover $\bO_\theta(F)^*=\bO_\theta(\overline F)$, we deduce by  Corollary \ref{wignersup}:
$$
\big|\big\langle\psi,\bO_\theta(F)\vf\big\rangle\big|\leq \|\vf\|_2\|\psi\|_2 \big(q^{n}+\|
   \mu_0^{s+1}\|_1\big)\, \|J^{-s} F\|_{\infty}.
   $$
   Hence, we get the following statement :
\begin{thm}
\label{caldaf}
Let $\theta\in\CO_\k^\times$.
The quantization map 
$\bO_\theta:\CD'(X_n)\to\CL\big(\CD(U_n),\CD'(U_n)\big)$ restricts to a continuous linear mapping from
the symbol space $\CB(X_n)$ to the space of bounded operators $\CB(L^2(U_n))$. More precisely, for each $s<-2$ and $F \in \CB(X_n)$, 
we have the operator norm estimate:
\begin{equation*}
\| {\bf \Omega}_{\theta}(F) \| \leq  \big(q^{n}+\|
   \mu_0^{s+1}\|_1\big)\, \, \|J^{-s} F\|_{\infty}.  
\end{equation*}
\end{thm}

\subsection{The composition law of symbols in $\CB(X_n)$}
\label{DD}

The results of section \ref{CLSL2} are, of course, not directly applicable for symbols in $\CB(X_n)$. 
In particular, we don't have the integral formula given in Proposition \ref{IF}.
Instead of trying to extend from $\CD(X_n)$ to $\CB(X_n)$ this oscillatory integral formula
(which is probably feasible), we obtain here a composition result entirely  based on Wigner functions.  
Our arguments  follow closely \cite[Section 6 \& 8]{Un84}. 

We first need ``non-integrated'' versions of the crucial estimates given in the previous section. For convenience, we shall introduce the following family of numerical
functions:
$$
\omega_s:G_n\to \R_+,\quad 
g=(u,t)\mapsto \mu_0^{s}(t)+\mathds 1_{\varpi^{-n}\CO_\k}(t),\qquad s\in\R.
$$
Note that $\omega_s$ is invariant under group inversion and, moreover, we have for all $s\in\R$, all $g=(u,t)\in G_n$:
$$
 \mu_0^{s}(t)\leq \omega_s(g)\leq 2\,\mu_0^{s}(t).
 $$
 Hence, we have the Peetre type inequalities:
 \begin{align}
 \label{P2}
 \omega_s(gg')\leq 
2\,\omega_s(g) \times\begin{cases}
 \omega_s(g')\quad &\mbox{if}\quad s\geq 0\\
\omega_{-s}(g')\quad &\mbox{if}\quad s< 0
 \end{cases},\quad \forall g,g'\in G_n.
 \end{align}
 
\begin{lem}
\label{ma}
For $\theta\in\CO_\k^\times$ and $s\in\R$,
 we have
 \begin{align}
 \label{T1}
 \int_{ X_n} \big | J^{s}W_{g_1}^{\theta}([g_2])   \big  | \,d[g_2]
 \leq q^{-3n}\,\omega_{s+1}(g_1),\quad \forall s \in\R.
 \end{align}
Moreover, we have for $F\in\CB(X_n)$ and for  $s<-1$ arbitrary:
 \begin{align}
 \label {T_2}
  \big|\big\langle \pi_{\theta}(g_1)\mathds{1}_{U_n}, \,\bO_{\theta}(F) 
\pi_{\theta}(g_2) \mathds{1}_{U_n} \big\rangle\big| \leq q^{-2n}\,\|J^{-s}F\|_\infty\,\omega_{s+1}(g_1^{-1}g_2).
\end{align}
\end{lem}
\begin{proof}
The first inequality is obtained using exactly  the same arguments than those given in Proposition \ref{maj}, but without performing the integral
over $g_1\in G_n$ (which is the reason why we don't need the constraint $s<-2$ here).
The second inequality follows from the first inequality combined with the formula \eqref{formula} (formula which is valid for every $s\in\R$,
but which is trivial in the case $s\geq -1$).
\end{proof}

The next result is a kind of converse of the estimate \eqref{T_2}:
\begin{prop}
\label{estim}
Let $\theta\in\CO_\k^\times$.  
Let $A$ be a bounded operator on $L^2(U_n)$ such that (with the notation of Lemma \ref{ma}) for all
$s<0$ there exists $C_s>0$ such that for all $g_1,g_2\in G_n$, we have
$$
 \big|\big\langle \pi_{\theta}(g_1)\mathds{1}_{U_n}, \,A\,
\pi_{\theta}(g_2) \mathds{1}_{U_n} \big\rangle\big| \leq C_s\,\omega_{s}(g_1^{-1}g_2).
$$
Then, there exists a unique function $F_A\in \CB(X_n)$ such that $A=\bO_{\theta}(F_A)$.
\end{prop}
\begin{proof}
Consider the function $F_A$ on $X_n$ defined by:
$$
F_A([g]):=q^{2n}\int_{G_n\times G_n} \big\langle \pi_{\theta}(g_1)\mathds{1}_{U_n}, \,A\,
\pi_{\theta}(g_2) \mathds{1}_{U_n} \big\rangle\,\overline{W^\theta_{g_1^{-1}g_2}}([g_1^{-1}g])\,dg_1dg_2.
$$
Let us first check that $F_A$ belongs to the symbol space $\CB(X_n)$. Since the operator $J$ commutes
with left translations, we get for $j\in\N$:
$$
J^jF_A([g]):=q^{2n}\int_{G_n\times G_n} \big\langle \pi_{\theta}(g_1)\mathds{1}_{U_n}, \,A\,
\pi_{\theta}(g_2) \mathds{1}_{U_n} \big\rangle\,\overline{J^{-j}W^\theta_{g_1^{-1}g_2}}([g_1^{-1}g])\,dg_1dg_2.
$$
By assumption, we have for  $s<0$ arbitrary:
\begin{align*}
&q^{-2n}|J^{j}F_A([g])|\leq \int_{G_n\times G_n} \big|\big\langle \pi_{\theta}(g_1)\mathds{1}_{U_n}, \,A\,
\pi_{\theta}(g_2) \mathds{1}_{U_n} \big\rangle\big|\,\big|\overline{J^{j}W^\theta_{g_1^{-1}g_2}}([g_1^{-1}g])\big|\,dg_1dg_2\\
&\leq C_{s} \int_{G_n\times G_n} \omega_{s}(g_1^{-1}g_2)\,\big|\overline{J^{j}W^\theta_{g_1^{-1}g_2}}([g_1^{-1}g])\big|\,dg_1dg_2
= C_{s} \int_{G_n\times G_n} \omega_{s}(g_1^{-1}g_2)\,\big|\overline{J^{j}W^\theta_{g_1^{-1}g_2}}([g_1^{-1}])\big|\,dg_1dg_2\\
&\quad= C_{s} \int_{G_n\times G_n} \omega_{s}(g_2)\,\big|\overline{J^{j}W^\theta_{g_2}}([g_1^{-1}])\big|\,dg_1dg_2
= C_{s} \int_{G_n\times G_n} \omega_{s}(g_2)\,\big|\overline{J^{j}W^\theta_{g_2}}([g_1])\big|\,dg_1dg_2\\
&\quad= C_{s} q^n\int_{X_n\times G_n} \omega_{s}(g_2)\,\big|\overline{J^{j}W^\theta_{g_2}}([g_1])\big|\,d[g_1]dg_2
\leq C_{s} q^{-2n} \int_{G_n} \omega_{s}(g_2)\,\omega_{j+1}(g_2)dg_2,
\end{align*}
which is finite for $s<-2-j$. 

It remains to show that $\bO_\theta(F_A)=A$, 
unicity will then follow from  injectivity of the quantization map $\bO_\theta$ (see Remark \ref{INJ}). Since the set $\{\pi_{\theta}(g)\mathds{1}_{U_n}: g\in G_n\}$
is total in $L^2(U_n)$, it suffices to show that for all $g_1,g_2\in G_n$, we have:
$$
 \big\langle \pi_{\theta}(g_1)\mathds{1}_{U_n}, \,A\,\pi_{\theta}(g_2) \mathds{1}_{U_n} \big\rangle=
  \big\langle \pi_{\theta}(g_1)\mathds{1}_{U_n}, \,\bO_\theta(F_A)\,\pi_{\theta}(g_2) \mathds{1}_{U_n} \big\rangle.
  $$
By Proposition \ref{extend} and Fubini's Theorem (all the integrals are absolutely convergent here), we get:
\begin{align*}
 & \big\langle \pi_{\theta}(g_1)\mathds{1}_{U_n}, \,\bO_\theta(F_A)\,\pi_{\theta}(g_2) \mathds{1}_{U_n} \big\rangle
  =q^n\int_{X_n} F_A([g]) \,W_{g_1^{-1} g_2}^{\theta}([g_1^{-1}g])\,d[g]\\
   &\qquad=q^{3n}\int_{X_n\times G_n\times G_n} \big\langle \pi_{\theta}(g_1')\mathds{1}_{U_n}, \,A\,
\pi_{\theta}(g_2') \mathds{1}_{U_n} \big\rangle\,\overline{W^\theta_{g_1'^{-1}g_2'}}([g_1'^{-1}g])\,W_{g_1^{-1} g_2}^{\theta}([g_1^{-1}g])\,d[g]dg_1'dg_2' .
  \end{align*}
  Now, by Remark \ref{rm} (ii), we have (with Dirac's ket-bra notation)
  \begin{align*}
  &q^{3n}\int_{X_n} \overline{W^\theta_{g_1'^{-1}g_2'}}([g_1'^{-1}g])\,W_{g_1^{-1} g_2}^{\theta}([g_1^{-1}g])\,d[g]\\
  &\qquad=
  q^{n}\int_{X_n} \overline{\bO_\theta^*\big(\big|\pi_{\theta}(g_2') \mathds{1}_{U_n} \big\rangle
   \big\langle \pi_{\theta}(g_1')\mathds{1}_{U_n}\big|\big)}\,
    \bO_\theta^*\big(\big|\pi_{\theta}(g_2) \mathds{1}_{U_n} \big\rangle
   \big\langle \pi_{\theta}(g_1)\mathds{1}_{U_n}\big|\big)\,d[g] \\
& \qquad  = q^{2n}\big\langle \pi_{\theta}(g_2')\mathds{1}_{U_n},\pi_{\theta}(g_2) \mathds{1}_{U_n} \big\rangle
   \big\langle \pi_{\theta}(g_1)\mathds{1}_{U_n},\pi_{\theta}(g_1') \mathds{1}_{U_n} \big\rangle,
  \end{align*}
which follows from the unitarity of $q^{-n/2}\bO_\theta$. We conclude using twice the resolution of identity given in Corollary \ref{ident}.
\end{proof}

We finally show that the linear space of bounded operators $\{\bO_\theta(F)\,:\,F\in\CB(X_n)\}$, forms an algebra.

\begin{thm}
\label{alg}
 Let $\theta\in\CO_\k^\times$.
Then, for $F_1,F_2\in \CB(X_n)$, there exists  $F_3\in \CB(X_n)$ such that:
$$\bO_\theta(F_1)\bO_\theta(F_2)=\bO_\theta(F_3).$$
\end{thm}
\begin{proof}
By Proposition \ref{estim}, it suffices to show that for all $s<0$ there exists $C_s>0$ such that for all $g_1,g_2\in G_n$, we have
$$
 \big|\big\langle \pi_{\theta}(g_1)\mathds{1}_{U_n}, \,\bO_\theta(F_1)\bO_\theta(F_2)\,
\pi_{\theta}(g_2) \mathds{1}_{U_n} \big\rangle\big| \leq C_s\,\omega_{s}(g_1^{-1}g_2).
$$
By Corollary \eqref{ident}, we have
\begin{align*}
&\big\langle \pi_{\theta}(g_1)\mathds{1}_{U_n}, \,\bO_\theta(F_1)\bO_\theta(F_2)\,
\pi_{\theta}(g_2) \mathds{1}_{U_n} \big\rangle\\
&\qquad\qquad\qquad\qquad=q^n\int_{G_n}\big\langle \pi_{\theta}(g_1)\mathds{1}_{U_n}, \,\bO_\theta(F_1)\pi_{\theta}(g)\mathds{1}_{U_n}\big\rangle
\big\langle\pi_{\theta}(g)\mathds{1}_{U_n},\,\bO_\theta(F_2)\,
\pi_{\theta}(g_2) \mathds{1}_{U_n} \big\rangle dg.
\end{align*}
Now, we can use the estimate  \eqref{T_2}, to deduce for arbitrary $s_1,s_2<-1$:
\begin{align*}
&\big|\big\langle \pi_{\theta}(g_1)\mathds{1}_{U_n}, \,\bO_\theta(F_1)\bO_\theta(F_2)\,
\pi_{\theta}(g_2) \mathds{1}_{U_n} \big\rangle\big|\\
&\qquad\qquad\leq 
q^{-3n}\,\|J^{-s_1}F_1\|_\infty\,\|J^{-s_2}F_2\|_\infty\,\int_{G_n}\omega_{s_1+1}(g_1^{-1}g)\omega_{s_2+1}(g^{-1}g_2)\,dg\\
&\qquad\qquad\qquad=
q^{-3n}\,\|J^{-s_1}F_1\|_\infty\,\|J^{-s_2}F_2\|_\infty\,\int_{G_n}\omega_{s_1+1}(g)\omega_{s_2+1}(g^{-1}g_1^{-1}g_2)\,dg\\
&\qquad\qquad\qquad\leq
q^{-3n}\,\|J^{-s_1}F_1\|_\infty\,\|J^{-s_2}F_2\|_\infty\,\Big(\int_{G_n}\omega_{s_1+1}(g)\omega_{-s_2-1}(g^{-1})\,dg\Big) \;\omega_{s_2+1}(g_1^{-1}g_2),
\end{align*}
where the last inequality follows from the Peetre inequality \eqref{P2}. This is enough to conclude since the above integral is finite when $s_1-s_2<-1$.
\end{proof}

\section*{Acknowledgement} 
We thank the
referee for useful remarks and suggestions.

\end{document}